\documentclass[aos,preprint]{imsart}
\RequirePackage[OT1]{fontenc}
\RequirePackage{amsthm,amsmath}

\RequirePackage[numbers]{natbib}
\usepackage{pictexwd}
\usepackage{bm,amssymb,graphicx,url}
\usepackage{color}
\usepackage{epstopdf}
\usepackage[normalem]{ulem}

\usepackage[colorlinks,citecolor=blue,urlcolor=blue]{hyperref}

\startlocaldefs    

\def\b{\beta}

\def\R{\mathbb R}
\def\dd{\Delta}

\def\E{{\mathbb E}}

\def\G{{\mathbb G}}
\def\H{{\mathbb H}}
\def\IK{I\!\!K}

\def\labda1{\lambda_1}
\def\labda2{\lambda_2}

\def\e{\varepsilon}

\def\t{\tau}

\def\s{\sigma}

\def\comment#1{\relax}

\def\=in{\mathop{\rm =}}

\numberwithin{equation}{section}
\theoremstyle{plain}
\endlocaldefs

\usepackage{amsmath,here,amsfonts,latexsym,graphicx,here}
\usepackage{amsthm,curves}
\usepackage{graphicx}
\usepackage{caption}
\usepackage{subcaption}
\usepackage{float}

\def\G{{\mathbb G}}
\def\H{{\mathbb H}}

\def\G{{\mathbb G}}
\def\H{{\mathbb H}}

\def\dd{\Delta}

\newtheorem{theorem}{Theorem}
\newtheorem{lemma}{Lemma}

\newtheorem{remark}{Remark}

\newtheorem{example}{Example}

\begin{document}
\begin{frontmatter}
\title{Confidence intervals in monotone regression}
\runtitle{monotone regression}

\begin{aug}
\author{\fnms{Piet} \snm{Groeneboom}\corref{}\ead[label=e1]{P.Groeneboom@tudelft.nl}}
\and
\author{\fnms{Geurt} \snm{Jongbloed}\corref{}\ead[label=e2]{G.Jongbloed@tudelft.nl}}
\runauthor{Piet Groeneboom and Geurt Jongbloed}
\affiliation{Delft University of Technology}
\address{Delft Institute of Applied Mathematics, Mekelweg 4, 2628 CD Delft,
	The Netherlands.\\ 
	\printead{e1,e2}} 
\end{aug}

\begin{abstract}
We construct bootstrap confidence intervals for a monotone regression function. It has been shown that the ordinary nonparametric bootstrap, based on the nonparametric least squares estimator (LSE) $\hat f_n$ is inconsistent in this situation. We show, however, that a consistent bootstrap can be based on the smoothed $\hat f_n$, to be called the SLSE (Smoothed Least Squares Estimator).

The asymptotic pointwise distribution of the SLSE is derived. The confidence intervals, based on the smoothed bootstrap, are compared to intervals based on the (not necessarily monotone) Nadaraya Watson estimator and the effect of Studentization is investigated. We also give a method for automatic bandwidth choice, correcting work in \cite{SenXu2015}. The procedure is illustrated using a well known dataset related to climate change. 
\end{abstract} 



\end{frontmatter}


\maketitle

\section{Introduction}
\label{sec:Intro}
We consider the monotone regression setting where we observe independent pairs $(X_i,Y_i)$ of random variables ($1\le i\le n$), where the $X_i$ are i.i.d.\ with non-vanishing density $g$ on $[0,1]$ and 
\begin{align}
\label{regression_model}
Y_i=f_0(X_i)+\e_i, \,\,\, 1\le i\le n.
\end{align}
The regression function $f_0:[0,1]\mapsto\R$ is nondecreasing or nonincreasing and the $\e_i$ are i.i.d.\ sub-Gaussian with expectation $0$ and variance $\s_0^2$, independent of the $X_i$'s. Our aim is to construct pointwise nonparametric confidence intervals for $f_0(t)$. In the discussion below we will only return to the nonincreasing functions in the  example, given in Section \ref{section:Mendota}, and for the rest stick to the nondecreasing functions.  The theory for these cases is completely similar.

The basic monotone least squared estimate (LSE) $\hat f_n$ of $f_0$ is the so-called isotonic regression of $(X_i,Y_i)$. This estimator
is defined as minimizer of
\begin{align*}
	\sum_{i=1}^n \bigl(Y_i-f(X_i)\bigr)^2
\end{align*}
over all nondecreasing functions $f$. This LSE can be computed via a straightforward method, using the so-called cumulative sum diagram (cusum diagram). From now on, we interpret $X_1,\ldots,X_n$ as {\it ordered} in the sense that $X_1< X_2\ldots< X_n$ and relabel the $Y_i$'s accordingly (as $Y_i$ related to the specific $X_i$).  The cusum diagram is then the set of points
\begin{align*}
(0,0),\qquad\Biggl(i,\sum_{j\le i}Y_j\Biggr)\qquad, i=1,\dots,n,
\end{align*}
and the monotone least squares estimator $\hat f_n(X_i)$ is given by the left-continuous slope of the greatest convex minorant of the cusum diagram evaluated at $i$ (see  Lemma 2.1 in \cite{piet_geurt:14}).

Constructing nonparametric pointwise bootstrap confidence intervals for $f_0(t)$  poses several difficulties. It has been proved by several authors that  the straightforward bootstrap, using resampling with replacement from the pairs $(X_i,Y_i)$ and computing the monotone least squares estimator $\hat f_n$ based the bootstrap samples, is inconsistent (see, e.g., \cite{kosorok:08}, \cite{SenXu2015} and \cite{sen_mouli_woodroofe:10} for results related to this phenomenon). It has been suggested in \cite{SenXu2015} in the context of interval censored data to generate  intervals by using a smoothed bootstrap. In the monotone regression setting, this means that one fixes the values of the $X_i$'s and generates bootstrap values of the $Y_i$ by resampling with replacement from the residuals of the $Y_i$ with respect to a smooth monotone estimate of $f_0$. 
This approach addresses the intrinsic cause of the inconsistency of the bootstrap method based on direct resampling of the pairs, namely the fact that the derivative $f_0'$ cannot be estimated directly by differentiating the (piecewise constant) monotonic least squares estimator $\hat f_n$.

Bootstrap confidence intervals based on the smoothed maximum likelihood estimator (SMLE) for the related current status model were proposed in \cite{kim_piet:17EJS} and \cite{kim_piet:18:SJS}. One of the main issues in these papers is the treatment of the bias, which will be treated differently in the present paper. The usual methods for dealing with the bias are undersmoothing and correction by estimating the bias. These methods are rather unsatisfactory for the present model. \cite{{marron:91}} suggests a third method, which we will systematically use in the sequel.

In section \ref{section:smoot_est}, we define a Smoothed Least Squares Estimator (SLSE) for $f_0$. We show that this estimator is asymptotically normally distributed with rate $n^{2/5}$ and derive its asymptotic bias and variance. Furthermore, we consider the smooth but not necessarily monotone Nadaraya Watson (NW) estimator of $f_0$. Based on the SLSE and the NW estimator, we propose bootstrap methods to construct confidence sets for $f_0(t)$ in section \ref{section:smoot_boot}. We also prove a theorem stating that the bootstrap method based on the SLSE asymptotically works, with specific choices for the various bandwidths involved. In particular, we show that the method for computing the optimal bandwidth in \cite{SenXu2015} will only work if   bandwidths of different order are chosen. Moreover, we empirically study the effect of Studentization on coverage probabilities. In section \ref{section:bandwidth_choice} we address the problem of bandwidth selection in practice. Also here, we propose a smoothed bootstrap approach. Section \ref{section:Mendota} illustrates the method using a well know dataset related to climate change. We note that the methods described in this paper can be applied using the open software \cite{piet_github:21}.

\section{Smooth (monotone) estimation of the regression function}
\label{section:smoot_est}
As immediately follows from its construction, the LSE $\hat f_n$ is a (non-smooth) step function. It can be used to define a smooth  estimate. We define a particular {\it Smoothed} LSE (SLSE), $\tilde f_{nh}$. For this, let $K$ be  a symmetric twice continuously differentiable nonnegative kernel with support $[-1,1]$ such that $\int K(u)\,du=1$. Let $h>0$ be a bandwidth and define the scaled kernel $K_h$  by
\begin{align}
	\label{def_K_h}
	K_h(u)=\frac1h K\left(\frac{u}{h}\right), \,\,\,u\in\R.
\end{align}
Then, for $t\in[h,1-h]$, the SLSE is defined by a two-sided local average of the LSE,
\begin{align}
\label{def_SLSE}
\tilde f_{nh}(t)=\int K_h(t-x)\,\hat f_n(x)\,dx.
\end{align}
Note that on  $[h,1-h]$ the SLSE inherits the monotonicity from $\hat f_n(t)$. For $t\notin[h,1-h]$ we use a boundary correction to be discussed later.
In this paper, we choose for $K$ the triweight kernel
\begin{align*}
	K(u)=\tfrac{35}{32}\left(1-u^2\right)^31_{[-1,1]}(u).
\end{align*}

Estimator (\ref{def_SLSE}) is rather different from the Nadaraya-Watson (NW) estimator, which is defined by
\begin{align}
\label{Nadaraya-Watson}
{\tilde f}_{nh}^{NW}(t)=\frac{\int y K_h(t-x)\,d\H_n(x,y)}{\int K_h(t-x)\,d\H_n(x,y)}=\frac{\sum_{i=1}^n Y_i K_h(t-X_i)}{\sum_{i=1}^n K_h(t-X_i)}\,,
\end{align}
for $t\in[h,1-h]$, where $\H_n$ is the empirical distribution function of the pairs $(X_i,Y_i)$, $i=1,\dots,n$. This estimator can also be differentiated, just as (\ref{def_SLSE}), but is not necessarily monotone.

The NW estimator (\ref{Nadaraya-Watson}) is seemingly simpler than SLSE (\ref{def_SLSE}), because it is expressed as sum over sample values, whereas (\ref{def_SLSE}) is an integral with respect to Lebesgue measure.
However, using integration by parts, (\ref{def_SLSE}) can also be rewritten as a simple sum. Indeed, for $t\in[h,1-h]$,
\begin{align}
\label{SLSE_as_sum}
\tilde f_{nh}(t)=\hat f_n(0)+\int_{x\in(0,t+h]}\IK_h(t-x)\,d\hat f_n(x)=\hat f_n(0)+\sum_{\t_i\in(0,t+h]}\IK_h(t-\tau_i) p_i,
\end{align}
where
\begin{align}
\label{def_IK}
\IK_h(y)=\int_{-\infty}^yK_h(u)\,du=\int_{-\infty}^{y/h} K(u)\,du.
\end{align}
Here the $\t_i$ are the locations of the jumps of the LSE $\hat f_n$ and the $p_i>0$ the sizes of the jumps.
This time the summation is over  points at stochastic locations $\t_i$ with stochastic masses $p_i$, characterizing the LSE. 

It is well known that defining $\tilde f_{nh}$ as in (\ref{def_SLSE})  outside the interval $[h,1-h]$, leads to  serious bias. We define $\tilde f_{nh}$ threfore slightly differently on the intervals near zero and one, using quadratic Taylor approximations at the points $h$ and $1-h$ respectively. 
More precisely, for  $t\in[0,h)$ we define
\begin{align}
	\label{SMLE2}
	&\widetilde f_{nh}(t)=\widetilde f_{nh}(h)+(t-h)\widetilde f_{nh}^{\prime}(h)+\tfrac12(t-h)^2\widetilde f_{n,h_0}^{\prime\prime}(h_0),
\end{align}
where
\begin{align*}
\widetilde f_{n,h_0}^{\prime\prime}(h_0)=\int K_{h_0}'(h_0-x)\,d\hat f_n(x)=\sum_{\t_i}K_{h_0}'(h_0-\t_i)p_i,
\end{align*}
where $h_0\asymp n^{-1/9}$ and the $\t_i$ are the points of jump of $\hat f_n$ with values $p_i$. Note that $\widetilde f_{n,h_0}^{\prime\prime}(h_0)$ converges in probability to to $f_0''(0+)$, as $n\to\infty$. We can replace $\widetilde f_{n,h_0}^{\prime\prime}(h_0)$ by any other consistent estimate of $f''(0+)$.

For $t\in[1-h,1]$ we define analogously
\begin{align}
	\label{SMLE3}
	&\widetilde f_{nh}(t)=\widetilde f_{nh}(h)+(t-(1-h))\widetilde f_{nh}^{\prime}(1-h)+\tfrac12(t-(1-h))^2\widetilde f_{n,h_0}^{\prime\prime}(1-h_0).
\end{align}
Note that we may lose the monotonicity in the boundary intervals $[0,h]$ and $[1-h,1]$  in this way, but that the probability of this phenomenon will tend to zero with increasing sample size.

For the NW estimator we use a different boundary correction, replacing the kernel $K$ by a linear combination of the kernels $K$ and $uK(u)$, because of its more complicated expression as a ratio. This type of boundary correction is for example described on p.\ 210 and 211 of \cite{piet_geurt:14}. 

We have the following asymptotic pointwise result for the SLSE. Note that by going from $\hat f_n(t)$ to (\ref{def_SLSE})
we improve the rate of convergence $n^{1/3}$ to $n^{2/5}$ and lose the ``non-standard asymptotics'' behavior of $\hat f_n$ (for its cube root $n$ convergence to Chernoff's distribution, see \cite{brunk:70}, Theorem 5.2, p. 190).

\begin{theorem}
\label{th:limit_SLSE}
	Let $f_0$ be a nondecreasing continuous function on $[0,1]$. Let $X_1,X_2,\ldots$ be i.i.d random variables with continuous density $g$, staying away from zero on $[0,1]$, and where the derivative $g'$ is continuous and bounded on $(0,1)$.
	Furthermore, let $\e_1,\e_2\ldots$ be i.i.d. random variables distributed according to a sub-Gaussian distribution with expectation zero and variance $0<\s_0^2<\infty$, independent of the $X_i$'s.	
	Then consider $Y_i$, defined by
	\begin{align*}
		Y_i=f_0(X_i)+\e_i, \,\,\,\, i=1,2,\ldots
	\end{align*}
	Suppose $t\in(0,1)$ such that $f_0$ has a strictly positive derivative and a continuous second derivative $f_0''(t)\ne0$ at $t$. Then, for the SLSE $\tilde f_{nh}$ defined by (\ref{def_SLSE}) based on the pairs $(X_1,Y_1),\ldots,(X_n,Y_n)$, and $h\sim cn^{-1/5}$ for $c>0$,
	\begin{align*}
		n^{2/5}\left\{\tilde f_{nh}(t)-f_0(t)\right\}\stackrel{{\cal D}}\longrightarrow N(\b,\s^2).
	\end{align*}
	Here
	\begin{align}
		\label{asymp_bias_var}
		\b=\tfrac12c^2 f_0''(t)\int u^2K(u)\,du\,\,\, \mbox{ and }\,\,\,\s^2=\frac{\s^2_{0}}{c g(t)}\int K(u)^2\,du.
	\end{align}
The asymptotically Mean Squared Error optimal constant $c$ is given by:
\begin{align*}
c=\left\{\frac{\s^2_{0}}{g(t)}\int K(u)^2\,du\Bigm/\left\{f_0''(t)\int u^2K(u)\,du\right\}^2\right\}^{1/5}.
\end{align*}
\end{theorem}

\begin{remark}
{\rm
For Example 1, discussed below, the conditions of Theorem \ref{th:limit_SLSE} are satisfied, and the asymptotically optimal bandwidth is approximately $0.7n^{-1/5}$ (in particular not depending on $t$). It is seen from Theorem \ref{th:limit_SLSE} that the smoothed LSE has the same rate of convergence and also the same asymptotic variance as the NW estimator under the usual conditions (see (\ref{asymp_var_NW}) below).
}
\end{remark}

We now give a road map of the proof of Theorem \ref{th:limit_SLSE}. The proof itself is given in Section \ref{section:appendix}. 
Since the estimators are based on the non-linear LSE $\hat f_n$, the proof is totally different from the proofs for the NW estimator. We have to introduce methods similar to those used in \cite{piet_geurt_birgit:10} for local smooth functionals in the current status model. 

The first step is to write
\begin{align}
	\label{eq:biascent}
	\tilde f_{nh}(t)-f_0(t)=\int K_h(t-x)\{\hat f_n(x)-f_0(x)\}\,dx+\int K_h(t-x)f_0(x)\,dx-f_0(t)
\end{align}
and represent the first term at the right hand side as functional
\begin{align*}
\int \psi_{t,h}(x)\left\{\hat f_n(x)-y\right\}\,\,dH_0(x,y), \mbox{ where }\psi_{t,h}(x)=\frac{K_h(t-x)}{g(x)}
\end{align*}
and $H_0$ is the distribution function of the pairs $(X_i,Y_i)$.
Next,  a piecewise constant version $\bar\psi_{t,h}$ of $\psi_{t,h}$ is constructed to be able to use the characterization of $\hat f_n$ as a least squares estimator, enabling us to write
\begin{align*}
\int \bar\psi_{t,h}(x)\left\{\hat f_n(x)-y\right\}\,\,dH_0(x,y)=\int \bar\psi_{t,h}(x)\left\{\hat f_n(x)-y\right\}\,\,d\bigl(H_0-\H_n\bigr)(x,y).
\end{align*}
Here $\H_n$ is the empirical distribution function of the pairs $(X_i,Y_i)$. The latter expression turns out to behave as the empirical integral
\begin{align*}
\int \psi_{t,h}(x)\left\{f_0(x)-y\right\}\,\,d\bigl(H_0-\H_n\bigr)(x,y),
\end{align*}
for which we have a central limit theorem, after multiplying with $n^{2/5}$ and letting $h\sim cn^{-1/5}$.

The main effort goes into showing that the remainder terms are of lower order. For example, it needs to be shown that
\begin{align*}
\int \left\{\bar\psi_{t,h}(x)-\psi_{t,h}(x)\right\}\left\{\hat f_n(x)-f_0(x)\right\}g(x)\,\,dx=o_p\bigl(n^{-2/5}\bigr).
\end{align*}
To this end we use the Cauchy-Schwarz inequality and the inequality
\begin{align*}
\left|\bar\psi_{t,h}(x)-\psi_{t,h}(x)\right|\le Kh^{-2}\left|\hat f_n(x)-f_0(x)\right|,
\end{align*}
for a $K>0$, which follows from a judicious choice of $\bar\psi_{t,h}$. We also need $L_2$ bounds for $\hat f_n-f_0$, restricted to an interval $[a,b]\subset(0,1)$, These follow from the condition that the errors $\e_i$ are sub-Gaussian and $L_2$-bounds for functions of uniformly bounded variation in Chapter 9 of \cite{geer:00}. Alternatively, we could use the ``switch relation'', as used in \cite{kim_piet:18:SJS} and \cite{SenXu2015}. The proof is completed by incorporating the behavior of the bias term in (\ref{eq:biascent}),
\begin{align*}
\int K_h(t-x)f_0(x)\,dx-f_0(t)=\tfrac12h^2f_0''(t) +o\left(h^2\right).
\end{align*}

\section{Confidence intervals based on the smoothed bootstrap}
\label{section:smoot_boot}
In this section, we will create confidence intervals for $f_0(t)$ based on the SLSE. As basis for the intervals, we choose the quantity
\begin{align}
	\label{eq:diffbase}
\tilde f_{nh}(t)-f_0(t)
\end{align}
as studied asymptotically in Theorem \ref{th:limit_SLSE}. The distribution of this quantity is approximated by the distribution of a related quantity based on observatons $(X_1,Y_1^*),\ldots,(X_n,Y_n^*)$ generated by an estimated model, which makes it a bootstrap approach. As an approximate (estimated) model, we choose to generate the $Y_i^*$-values by taking an {\it oversmoothed} SLSE, and adding noise to that. More precisely, we take $h_0\asymp n^{-1/9}$ (we will come back to this choice in Section \ref{section:bandwidth_choice}), compute $\tilde f_{nh_0}$ and also compute the residuals of the $Y_i$ with respect to this estimate:
\begin{align*}
E_i=Y_i-\tilde f_{nh_0}(X_i),\qquad i=1,\dots,n.
\end{align*} 
Next, we center the $E_i$ by substracting their mean:
\begin{align*}
\tilde E_i=E_i-n^{-1}\sum_{j=1}^n E_j, \qquad i=1,\dots,n.
\end{align*}

Using the $\tilde E_i$, we generate bootstrap samples
\begin{align}
\label{bootstrap1}
(X_i,Y_i^*)=\left(X_i,\tilde f_{nh_0}(X_i)+E_i^*\right),\qquad i=1,\dots,n,
\end{align}
where the $E_i^*$ are (discretely) uniformly drawn with replacement from the $\tilde E_i$, and consider the differences
\begin{align}
\label{SLSE_intervals0}
\tilde f_{nh}^*(t)-\tilde f_{nh_0}(t).
\end{align}
Here $\tilde f_{nh}^*(t)$ is the estimate of $f_0$, based on a bootstrap sample, with bandwidth $h$ as in (\ref{eq:diffbase}).
Note that we keep the $X_i$ fixed in the bootstrap samples.

\begin{example}
\label{example1}
{\rm
Consider the setting used in \cite{moumita:21}, who study a Bayesian approach to constructing confidence sets for $f_0$.  Take $f_0(x)=x^2+x/5$, $g(t)=1_{[0,1]}(t)$ and independent normal errors $\e_i$ with expectation $0$ and variance $0.01$. The NW estimator and the SLSE are shown as blue  solid curves  in Figure \ref{figure:CI_first_intervals100}.
The confidence intervals, of which the construction will be explained below, are shown in Figure \ref{figure:CI_first_intervals100} at the points $t=0.01,0.02,\dots,0.99$. The coverage is shown in Figure \ref{fig:percentages_naive_resampling100_first_function}. In Figure \ref{fig:percentages_naive_resampling100_second_function} we also show the results for the rather different function  $f_0(x)=\exp\{4(x -1/2)\}/\{1+\exp(4(x - 1/2)\}$, for which the second derivative is not constant. The results for sample size $n=500$ are given in Figure \ref{figure:coverage100_SLSE_NW_bootstrap_2nd_function}. For generating the confidence intervals and coverage percentages, we use  the code in {\color{red}  \cite{piet_github:21}}.
} 
\end{example}

\begin{figure}[!ht]
\begin{subfigure}[b]{0.45\textwidth}
\includegraphics[width=\textwidth]{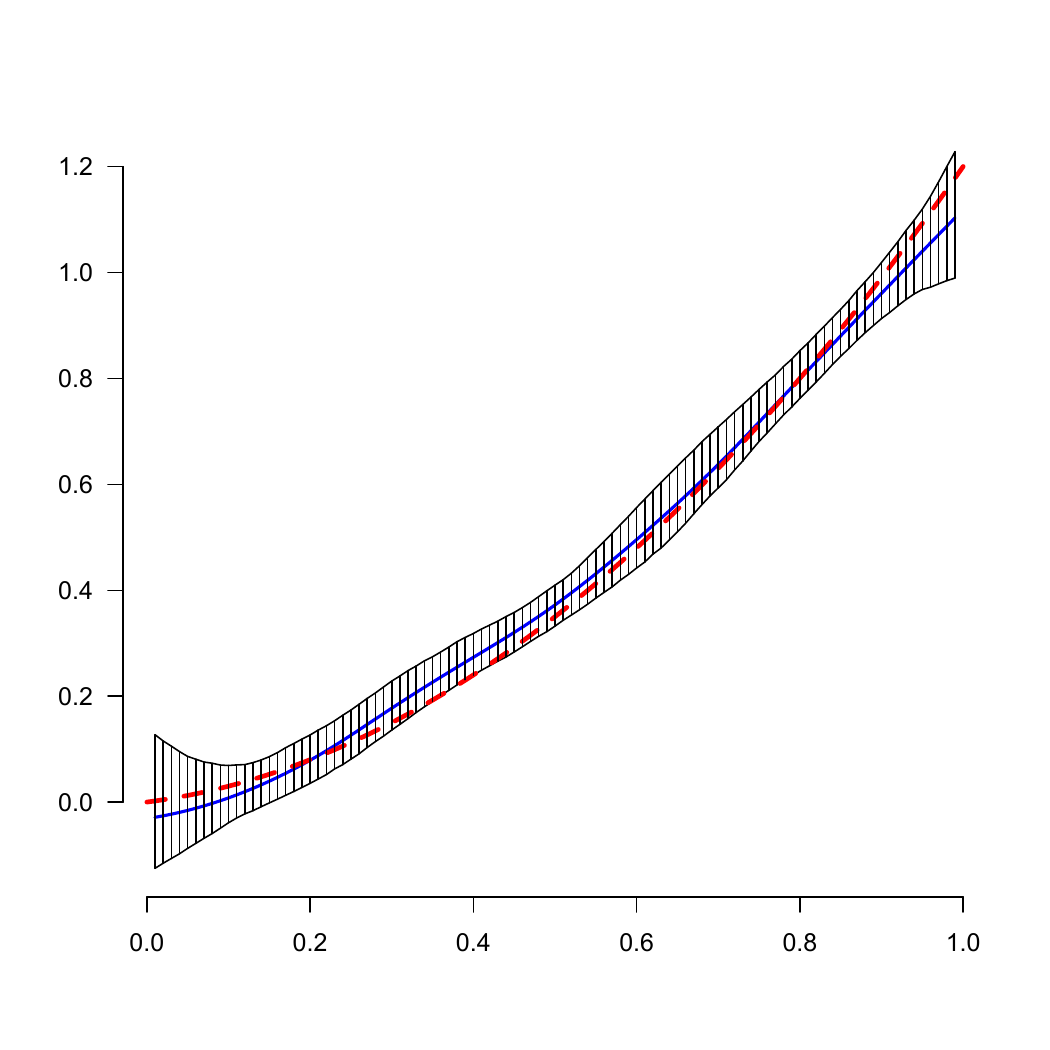}
\caption{}
\label{fig:CI_NW}
\end{subfigure}
\begin{subfigure}[b]{0.45\textwidth}
\includegraphics[width=\textwidth]{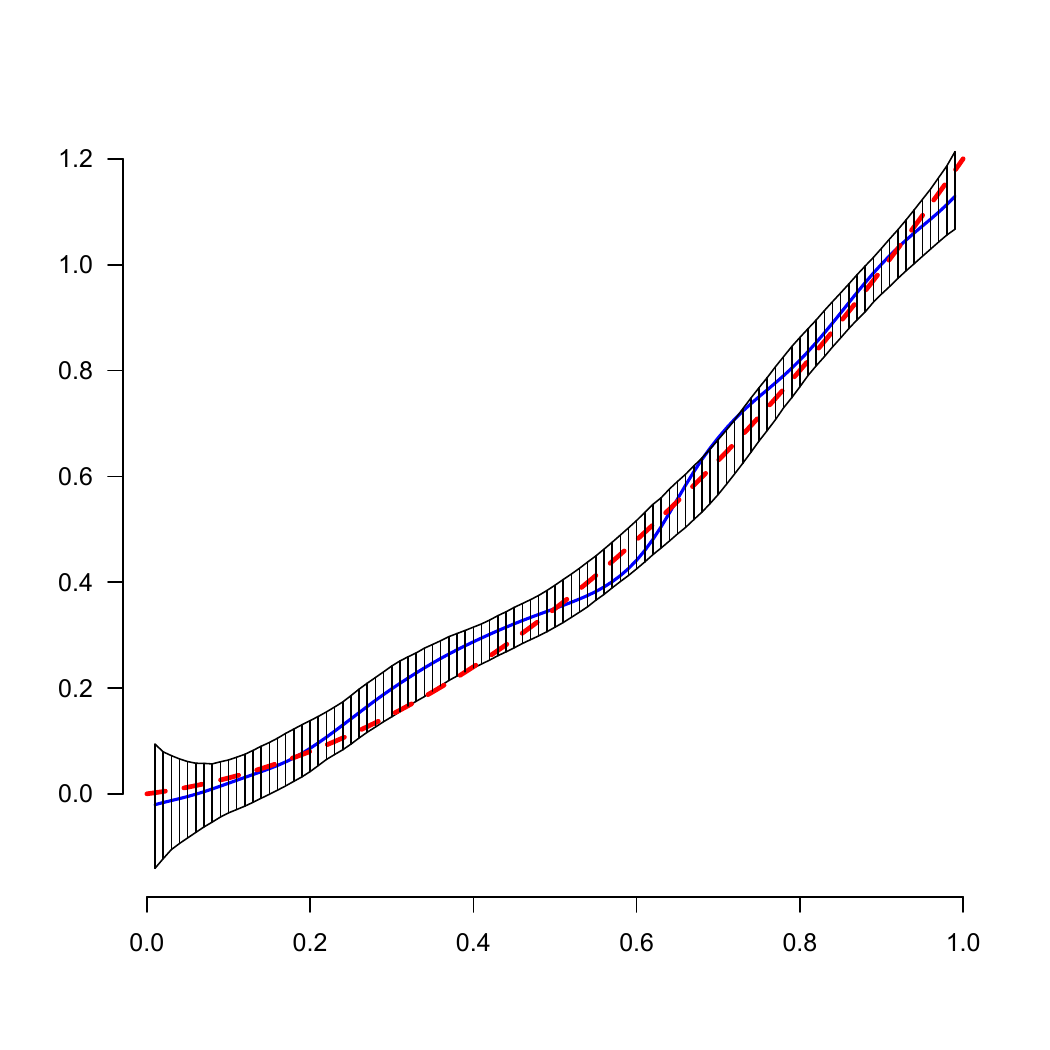}
\caption{}
\label{fig:CI_SLSE}
\end{subfigure}
\caption{(a) The SLSE (blue, solid) and $95\%$ confidence intervals for the same sample, using the confidence intervals (\ref{first_conf_intervals}). The  red dashed curve is $f_0$. (b) The NW estimator (blue, solid) and pointwise $95\%$ confidence intervals, for sample size $n=100$, the dashed red curve is the function $f_0$. In both cases the bandwidth $h=0.5n^{-1/5}$ and $h_0=0.7n^{-1/9}$.}
\label{figure:CI_first_intervals100}
\end{figure}

The $95\%$ bootstrap confidence intervals are given by
\begin{align}
	\label{first_conf_intervals}
	\left(\tilde f_{nh}(t)-Q_{0.975}^*,\tilde f_{nh}(t)-Q_{0.025}^*\right),
\end{align}
where $Q_{0.025}^*$ and $Q_{0.975}^*$ are the $2.5$th and $97.5$th percentiles of $1000$ (bootstrap) samples of (\ref{SLSE_intervals0}). 
Note that the percentiles $Q_{0.025}^*$ and $Q_{0.975}^*$ contain an estimate of the asymptotic bias
\begin{align*}
\tfrac12h^2\tilde f_{nh_0}''(t)\sim \tfrac12h^2f_0''(t),
\end{align*}
(see also Lemma \ref{lemma:bias_term} in Section \ref{section:bandwidth_choice}) and that therefore the bias of $\tilde f_{nh}(t)$ drops out in (\ref{first_conf_intervals}). So we do not need undersmoothing or explicit estimation of the bias in our procedure. The oversmoothing by taking $h_0\asymp n^{-1/9}$ (or at least a bandwidth tending to zero slower than $n^{-1/5}$) is essential here, though.

\begin{figure}[!ht]
\begin{subfigure}[b]{0.45\textwidth}
\includegraphics[width=\textwidth]{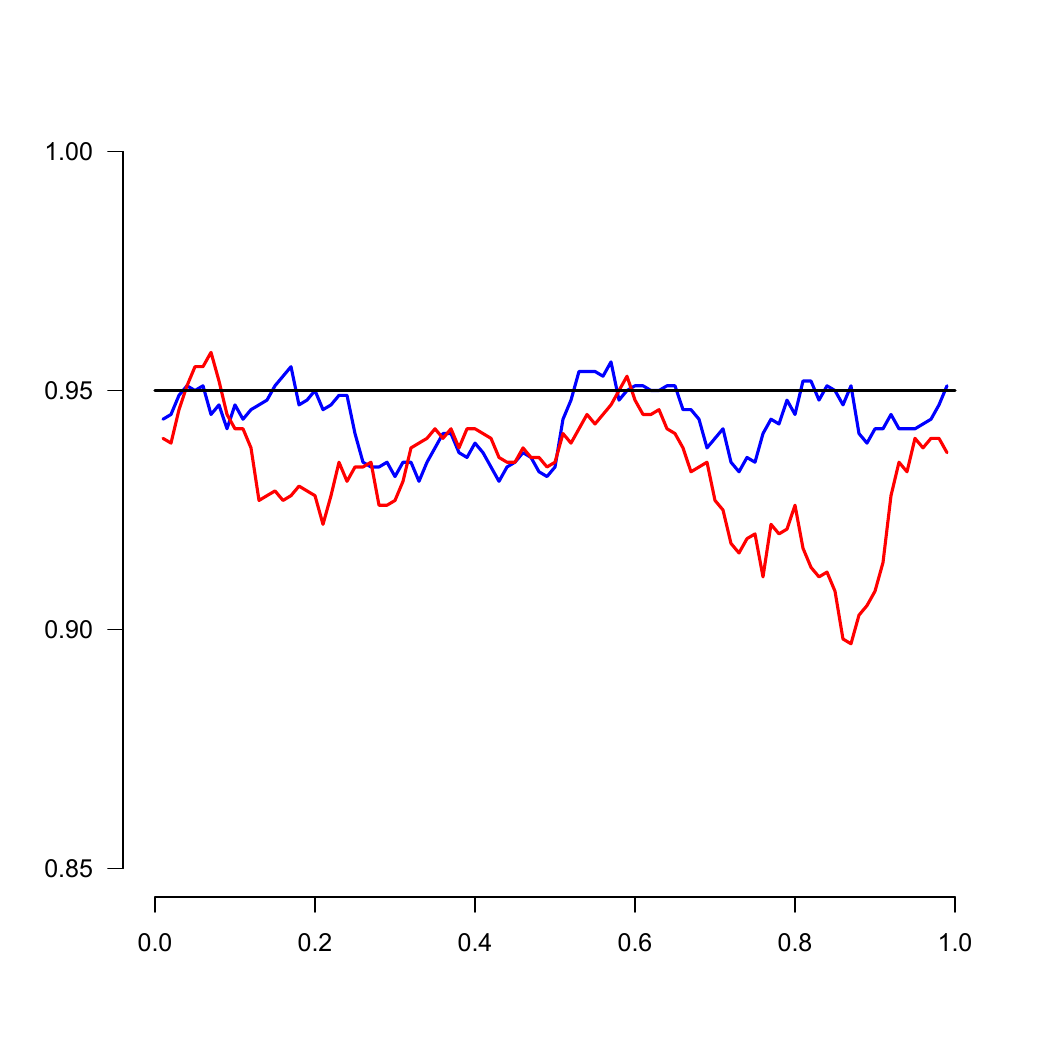}
\caption{}
\label{fig:percentages_naive_resampling100_first_function}
\end{subfigure}
\begin{subfigure}[b]{0.45\textwidth}
\includegraphics[width=\textwidth]{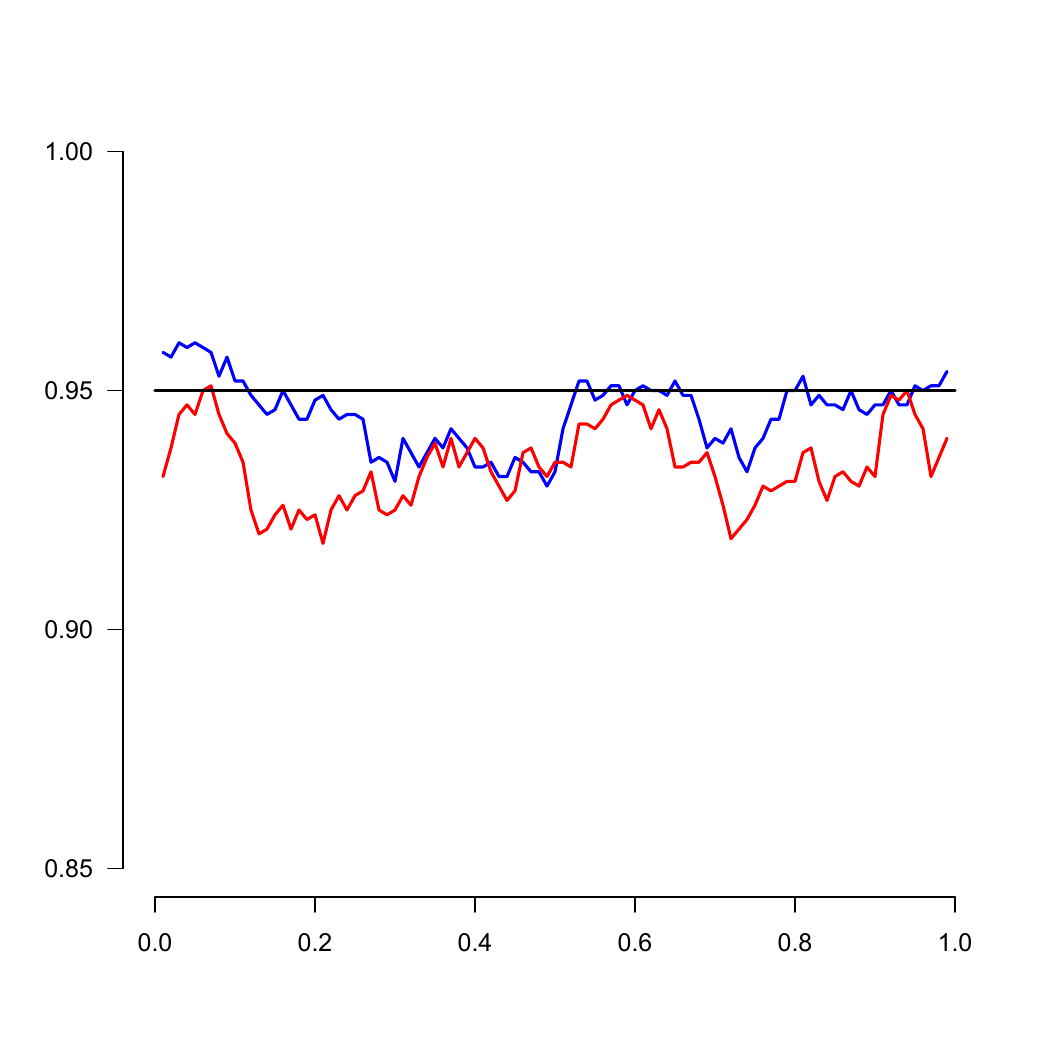}
\caption{}
\label{fig:percentages_naive_resampling100_second_function}
\end{subfigure}
\caption{Coverage of the confidence intervals, based on the bootstrap described. (a) Fraction of the $B$ experiments where $f_0(t)$ is in the interval (\ref{first_conf_intervals}) for  the function $f_0(x)=x^2+x/5$ and for the SLSE (blue) and  the NW estimator (red), and (b) fraction of the $B$ experiments where $f_0(t)$ is in the interval (\ref{first_conf_intervals}) for the function $f_0(x)=\exp\{4(x -1/2)\}/\{1+\exp(4(x - 1/2)\}$ and for the SLSE (blue) and the NW estimator (red), based on $B=1000$ samples of size $n=100$, and $t=0.01,\dots,0.99$. The chosen bandwidths $h$ and $h_0$ are $h=0.5n^{-1/5}$ and $h=0.7n^{-1/9}$.}
\label{figure:coverage100_SLSE_NW_bootstrap}
\end{figure}

Theorem \ref{th:bootstrap_SLSE} shows that asymptotically, the bootstrap method described will asymptotically give the right coverage, in the sense that after rescaling with $n^{2/5}$ the asymptotic distribution of (\ref{SLSE_intervals0}) under the estimated model (\ref{bootstrap1}) coincides with the asymptotic distribution of (\ref{eq:diffbase}) under model  (\ref{regression_model}). 

\begin{theorem}
\label{th:bootstrap_SLSE}
Let the conditions of Theorem \ref{th:limit_SLSE} be satisfied. Moreover, let $h\sim cn^{-1/5}$ and $h_0\sim c'n^{-1/9}$, for some positive constants $c$ and $c'$. Then, at $t\in(0,1)$,
\begin{align*}
n^{2/5}\left\{\tilde f_{nh}^*(t)-\tilde f_{nh_0}(t)\right\} \stackrel{\cal D}\longrightarrow N(\b,\s^2), 
\end{align*}
given $(X_1,Y_1),\dots,(X_n,Y_n)$, almost surely along sequences $(X_1,Y_1),(X_2,Y_2),\dots$, where $\b$ and $\s^2$ are defined in (\ref{asymp_bias_var}).
\end{theorem}

The proof is given in Section \ref{section:appendix}. It goes through similar steps as the proof of Theorem \ref{th:limit_SLSE} but is more complicated  because we are here in the ``bootstrap world'' and for example have to use the Lindeberg-Feller version of the central limit theorem to deal (conditionally) with the dependence on the changing regression function $\tilde f_{nh_0}$ instead of $f_0$.

We compare the confidence intervals based on the SLSE with confidence intervals based on the NW estimator. To construct the latter, we define the (empirical) residuals by
\begin{align*}
E_i=Y_i-\tilde f_{nh_0}^{NW}(X_i),\qquad i=1,\dots,n,
\end{align*}
where $\tilde f_{nh_0}^{NW}$ is the NW estimator with bandwidth $h_0$ (again of order $n^{-1/9}$), leading to the bootstrap quantity
\begin{align}
\label{NW_intervals0}
\left(\tilde f_{nh}^{NW}\right)^*(t)-\tilde f_{nh_0}^{NW}(t).
\end{align}
Here $\left(\tilde f_{nh}^{NW}\right)^*$ is the NW estimator based on (\ref{bootstrap1}) with $\tilde f_{nh_0}(X_i)$ replaced by  $\tilde f_{nh_0}^{NW}(X_i)$ and $E_i^*$ sampled with replacement from the residuals $\tilde E_i^{NW}$, $i=1,\dots,n$,
\begin{align*}
\tilde E_i^{NW}=E_i^{NW}-\bar E^{NW},\qquad E_i^{NW}=Y_i-\tilde f_{nh_0}^{NW}(X_i),\qquad \bar E^{NW}=n^{-1}\sum_{i=1}^n E_i^{NW}.
\end{align*} 

\begin{figure}[!ht]
\begin{subfigure}[b]{0.45\textwidth}
\includegraphics[width=\textwidth]{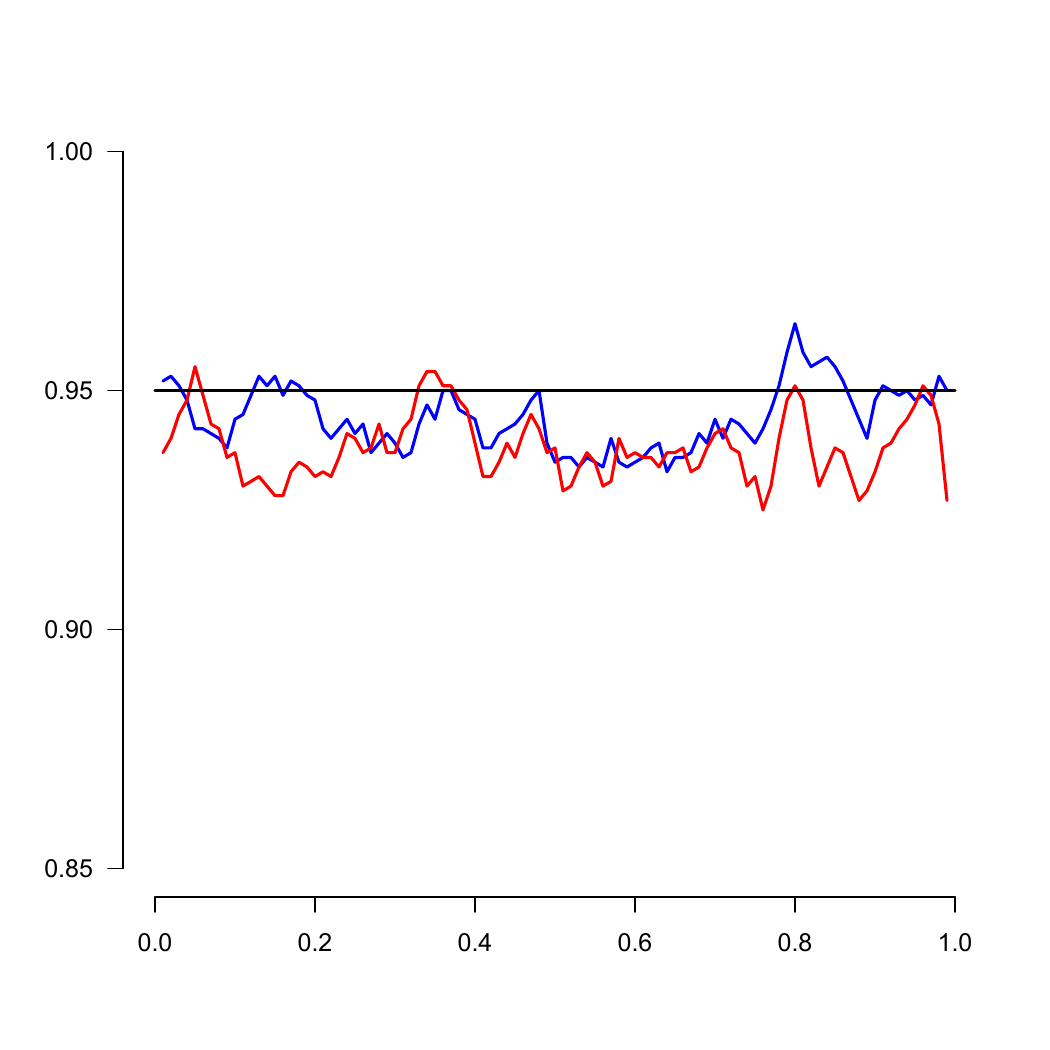}
\caption{}
\label{fig:percentage500_first_function}
\end{subfigure}
\begin{subfigure}[b]{0.45\textwidth}
\includegraphics[width=\textwidth]{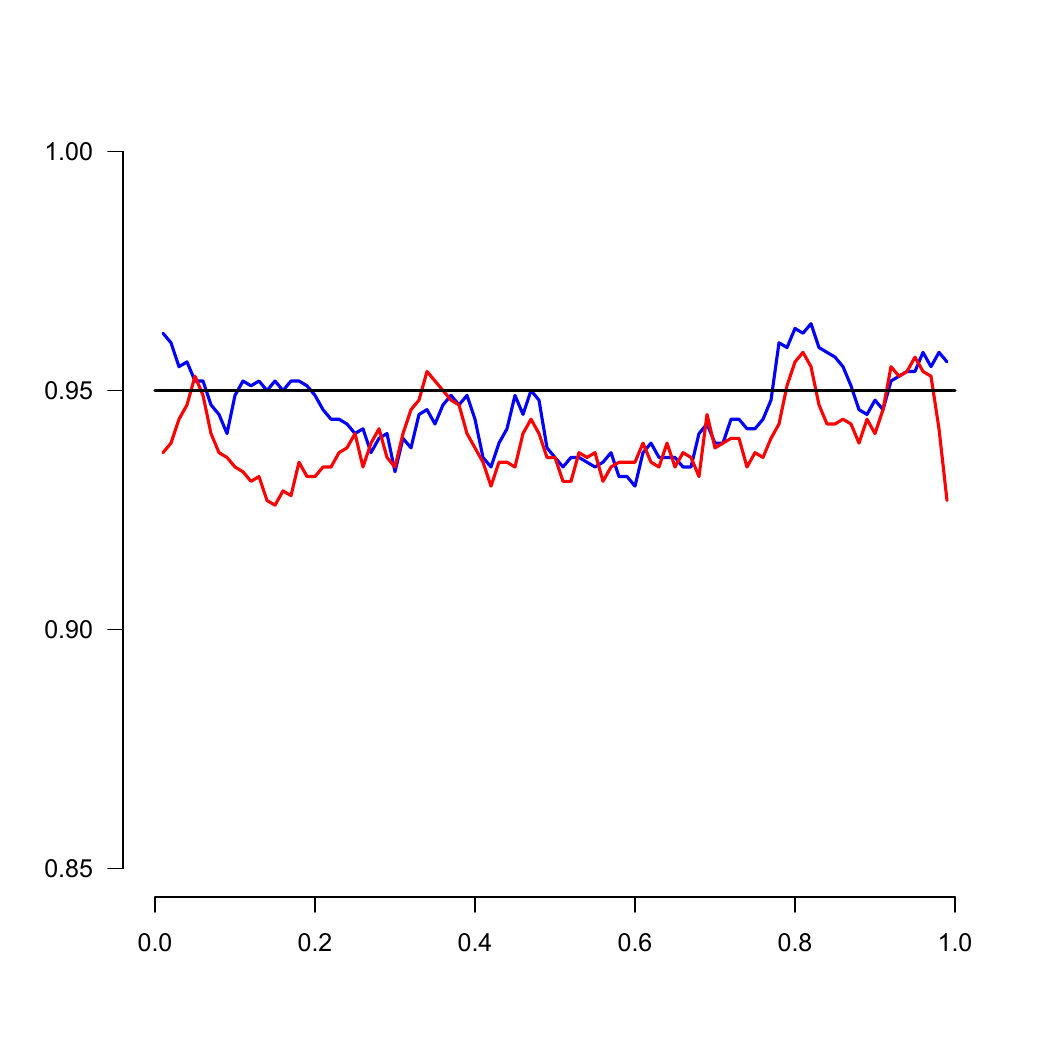}
\caption{}
\label{fig:percentages_naive_resampling500_second_function}
\end{subfigure}
\caption{Coverage of the confidence intervals, based on the bootstrap described. (a) Fraction of the $B$ experiments where $f_0(t)$ is in the intervals (\ref{first_conf_intervals}) for  the function $f_0(x)=x^2+x/5$ and for the SLSE (blue) and  the NW estimator (red), and (b) fraction of the $B$ experiments where $f_0(t)$ is in the intervals (\ref{first_conf_intervals}) for the function $f_0(x)=\exp\{4(x -1/2)\}/\{1+\exp(4(x - 1/2)\}$ and for the SLSE (blue) and the NW estimator (red), based on $B=1000$ samples of size $n=500$, and $t=0.01,\dots,0.99$. The chosen bandwidths $h$ and $h_0$ are $h=0.5n^{-1/5}$ and $h=0.7n^{-1/9}$.}
\label{figure:coverage100_SLSE_NW_bootstrap_2nd_function}
\end{figure}

In \cite{hall:92book} the variance of the NW estimator, conditionally on $X_1,\dots,X_n$, in the model (\ref{regression_model})  at $t$ is shown to be equal to
\begin{align}
\label{def_var_NW_Hall}
\t_t^2=\s_0^2\b_t^2,
\end{align}
where $\s_0^2=\text{var}(\e_i)$ and
\begin{align}
\label{beta_t_Hall}
\b_t^2=\frac{\sum_{i=1}^n K_h\left(t-X_i\right)^2}{\left\{\sum_{i=1}^n K_h(t-X_i)\right\}^2}\,.
\end{align}
In our set-up the parameter $\b_t$ is the same in the original sample and in the bootstrap samples, so to estimate $\tau_t^2$ we only need an estimate of $\s_0^2$.

Denoting the empirical distribution function of the $X_i$ by $\G_n$, note that
\begin{align*}
\frac1n\sum_{i=1}^n K_h\left(t-X_i\right)= \int K_h(t-x)\,d\G_n(x)\rightarrow^P g(t),
\end{align*}
whenever $g$ is continuous at $t$ and $h=h_n$ tends to zero such that $nh\rightarrow\infty$. Under the same conditions,
\begin{align*}
\frac{h}{n}\sum_{i=1}^n K_h\left(t-X_i\right)^2= h\int K_h(t-x)^2\,d\G_n(x)\rightarrow_P g(t)\int K(u)^2\,du.
\end{align*}
Therefore, as $n\rightarrow\infty$,
the variance of $\tilde f_{nh}^{NW}(t)$ behaves like 
\begin{align}
\label{asymp_var_NW}
\s_0^2\b_t^2=\sigma_0^2\frac{\frac1n\sum_{i=1}^n K_h\left(t-X_i\right)^2}{n\left\{\frac1n\sum_{i=1}^n K_h(t-X_i)\right\}^2}\sim\frac{\s_0^2}{nhg(t)}\int K(u)^2\,du=\frac{\s_0^2}{cn^{4/5}g(t)}\int K(u)^2\,du,
\end{align}
where we use $h=cn^{-1/5}$ in the final step. In view of  Theorem \ref{th:limit_SLSE}, it follows that the SLSE and the NW estimator (both rescaled and centered) have the same asymptotic variance.

A well known approach to improve the coverage of bootstrap confidence sets is Studentization. For the SLSE in the setting of this paper, this would mean that instead of using  difference (\ref{eq:diffbase}) as basis for the bootstrap, one would use a rescaled difference such that asymptotically the variance does not depend on unknown quantities anymore. In view of Theorem \ref{th:limit_SLSE}, this means
\begin{align}
	\label{eq:studentquantSLSE}
	\left\{\tilde f_{nh}(t)-f_0(t)\right\}/\hat\s_{n,0}
\end{align}
where the estimate  the variance $\s_0^2$, $\hat\s_{n,0}^2$,  is given by
\begin{align}
\label{var_estimate}
\hat\s_{n,0}^2=n^{-1}\sum_{i=1}^n\tilde E_i^2.
\end{align}
The distribution of (\ref{eq:studentquantSLSE}) under the true model is then approximated by the distribution of
\begin{align}
	\label{intervals_SLSE_Studentized}
	\left\{\tilde f_{nh}^*(t)-\tilde f_{nh_0}(t)\right\}/\hat\s_{n,0}^*
\end{align}
where
\begin{align*}
\left(\hat\s^*_{n,0}\right)^2=n^{-1}\sum_{i=1}^n\left(\tilde E_i^*-\bar E^*\right)^2,\qquad \bar E^*=n^{-1}\sum_{i=1}^n \tilde E_i^*,
\end{align*}
is the variance estimate based on a bootstrap sample, and where again $h_0\asymp n^{-1/9}$.  

A $95\%$ confidence interval for $f_0(t)$ can then be based on the $2.5$th and $97.5$th percentiles $Q_{0.025}^*$ and $Q_{0.975}^*$ of $1000$ bootstrap draws of (\ref{intervals_SLSE_Studentized}). It is then given by
\begin{align}
\label{second_conf_intervals_SLSE_Studentized}
\left(\tilde f_{nh}(t)-Q_{0.975}^*\hat\s_{n,0},\hat f_{nh}(t)-Q_{0.025}^*\hat\s_{n,0}\right),
\end{align}
where $\hat\s_{n,0}^2$ is defined by (\ref{var_estimate}). The comparison with the bootstrap intervals based on the SLSE without Studentization is shown in Figure \ref{fig:isotone+Studentized}.

For the NW estimator, an estimate of $\s_0^2$ is given on p. 226 of \cite{hall:92book} (but note the typo w.r.t.\ the index $j$ in \cite{hall:92book}). We take the definition from \cite{Hall_Kay:90} and define
\begin{align}
\label{sigma_NW}
\left(\hat\s_n^{NW}\right)^2=(n-m)^{-1}\sum_{i=1}^{n-2}\left(\sum_{j=0}^2 d_j Y_{i+j}\right)^2,
\end{align}
for the variance of the $\e_i$, where $m=2$ and
\begin{align*}
\left(d_0,d_1,d_2\right)=\left(\tfrac14(\sqrt{5}+1),-\tfrac12,-\tfrac14(\sqrt{5}-1)\right),
\end{align*}
(as recommended in \cite{Hall_Kay:90}).

\begin{figure}[!ht]
\begin{subfigure}[b]{0.45\textwidth}
\includegraphics[width=\textwidth]{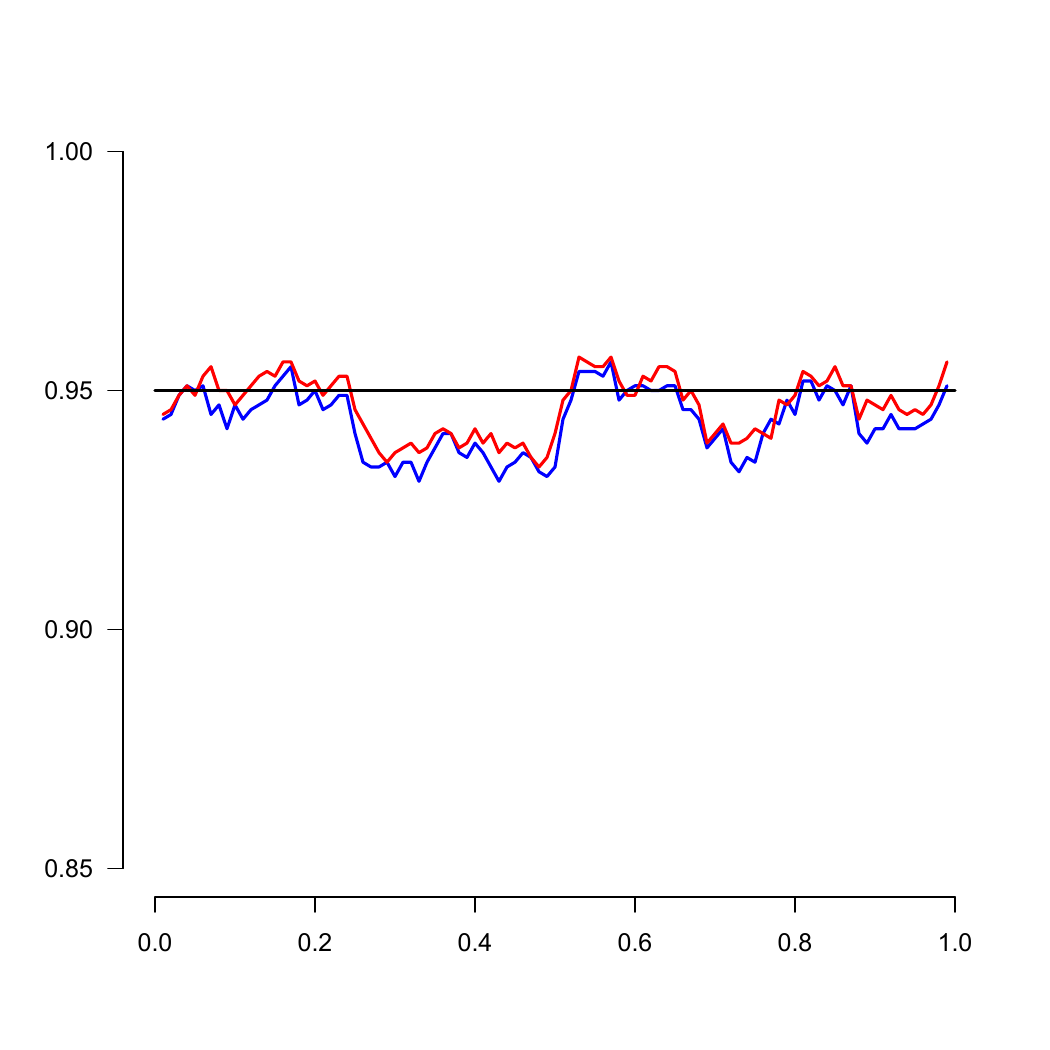}
\caption{}
\label{fig:isotone+Studentized}
\end{subfigure}
\begin{subfigure}[b]{0.45\textwidth}
\includegraphics[width=\textwidth]{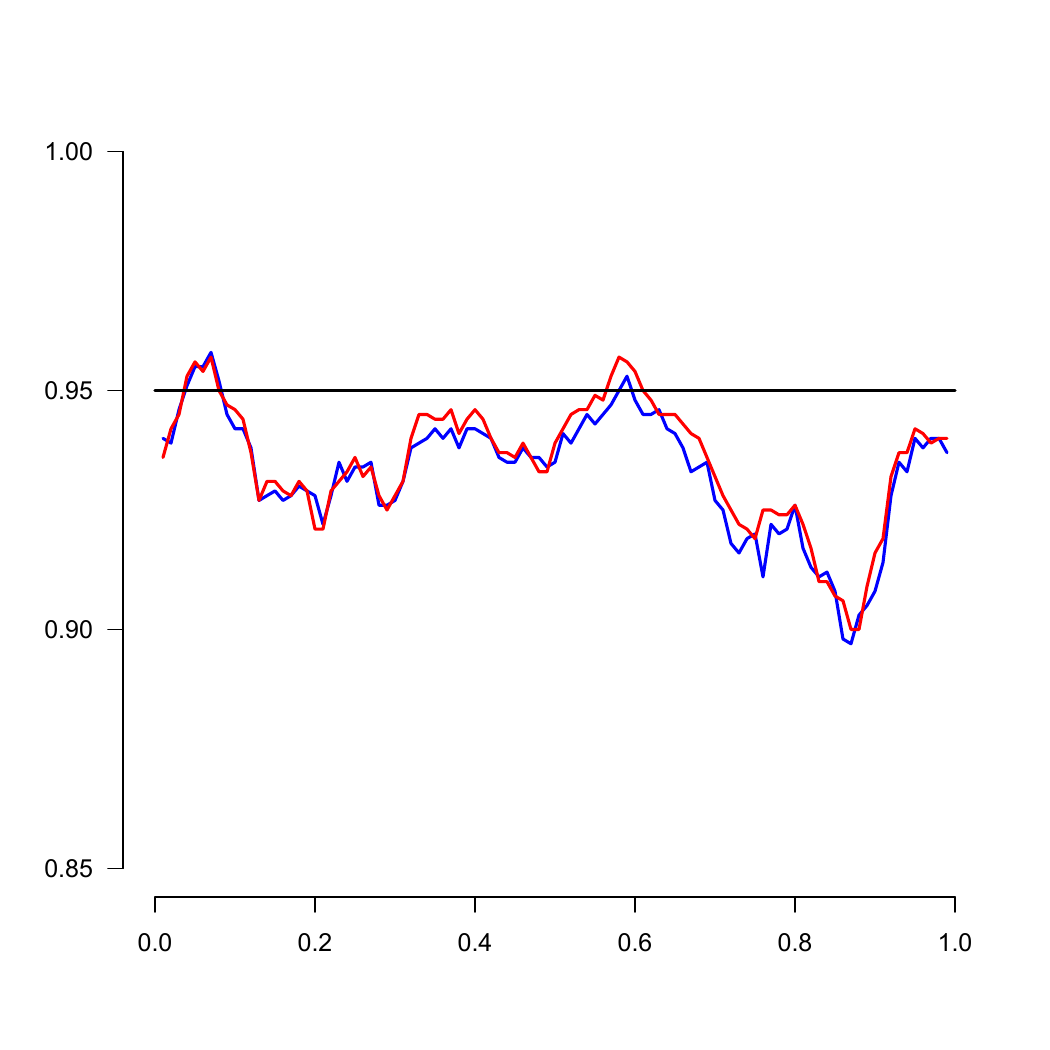}
\caption{}
\label{fig:NW+Studentized}
\end{subfigure}
\caption{Coverage of the confidence intervals, Studentized and non-Studentized, based on the bootstrap described. Fraction of the $B$ experiments where $f_0(t)$ is in the intervals (\ref{first_conf_intervals}) for  the function $f_0(x)=x^2+x/5$ for (a) the Studentized SLSE (red) and  the non-Studentized SLSE (blue), and (b) the Studentized NW estimator (red) and the non-Studentized NW estimator (blue). The figures are based on $B=1000$ samples of size $n=100$, and $t=0.01,\dots,0.99$. The chosen bandwidths $h$ and $h_0$ are $h=0.5n^{-1/5}$ and $h_0=0.7n^{-1/9}$.}
\label{figure:coverage100_Studentized_bootstrap}
\end{figure} 

If we now compare the non-Studentized and Studentized confidence intervals based on the NW estimator, constructed in the same way as in the case of the SLSE, we get Figure \ref{fig:NW+Studentized}
Here the non-Studentized are based on the differences (\ref{NW_intervals0})
and the Studentized intervals on the differences
\begin{align}
\label{intervals_NW2}
\left\{\left(\tilde f_{nh}^{NW}\right)^*(t)-\tilde f_{nh_0}^{NW}(t)\right\}/\left(\hat\s_n^{NW}\right)^*
\end{align}
where $\left(\hat\s_n^{NW}\right)^*$ is the estimate (\ref{sigma_NW}) for the bootstrap samples.
It is seen that in both cases there is not a great improvement.

For the Nadaraya-Watson estimator, one can also use the estimate of $\s_0$, based on the residuals, the type of estimate of $\s_0$ we used for the SLSE. In this case we get a bit more improvement for the Nadaraya-Watson estimator, see Figure \ref{figure:NW_Studentized2}. We still do not understand this phenomenon, based on the different ways of estimating $\s_0$ for the Nadaraya-Watson estimator, however.

\begin{figure}[!ht]
\centering
\includegraphics[width=0.45\textwidth]{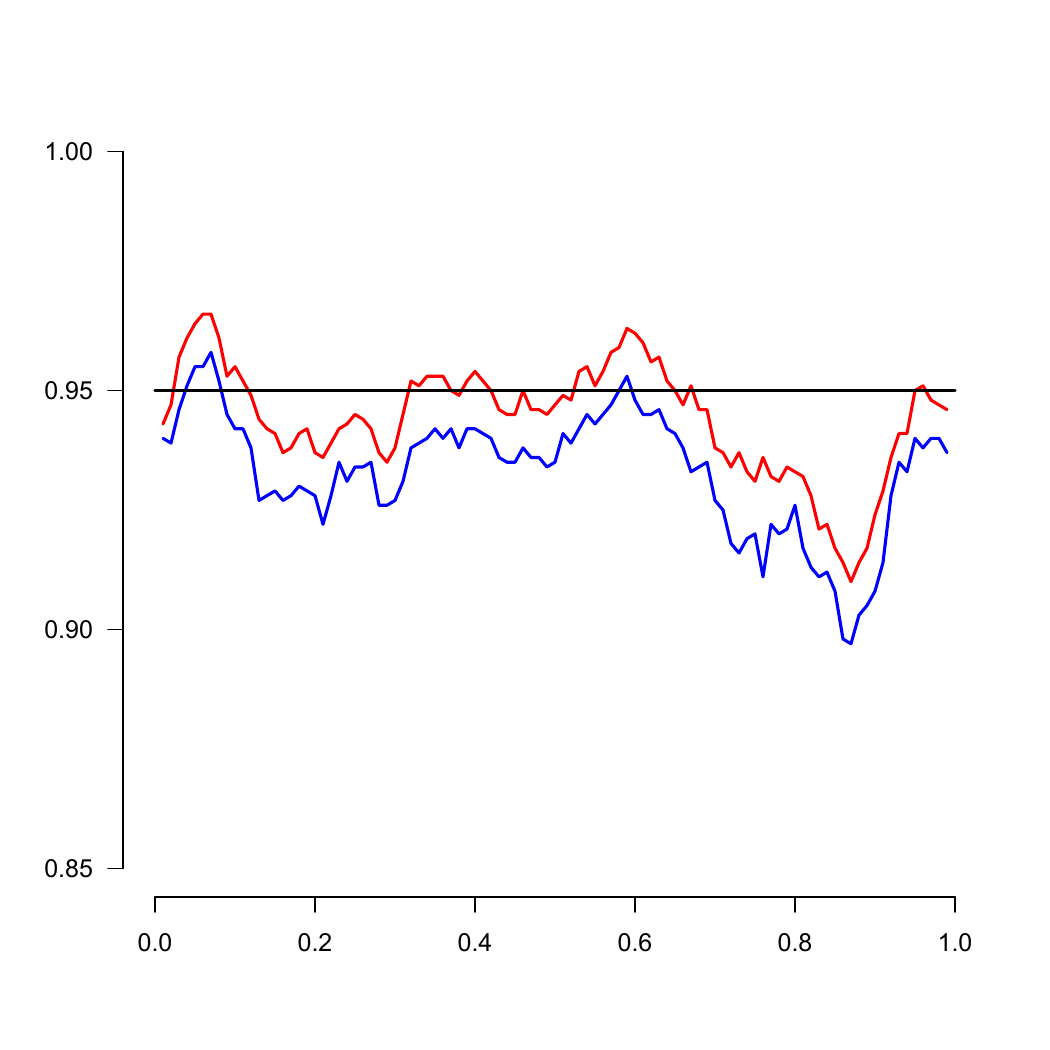}
\caption{The Studentized NW estimator (red), with variance estimated from the residuals, just as for the SLSE, and the non-Studentized NW estimator (blue). The figures are based on $B=1000$ samples of size $n=100$, and $t=0.01,\dots,0.99$. The chosen bandwidths $h$ and $h_0$ are $h=0.5n^{-1/5}$ and $h_0=0.7n^{-1/9}$.}
\label{figure:NW_Studentized2}
\end{figure}

\section{Bandwidth selection}
\label{section:bandwidth_choice}
We propose a bootstrap method to find an approximately MSE optimal bandwidth for estimating $f_0(t)$ at a point $t\in(0,1)$. The MSE we want to minimize as a function of $h$ is given by:
\begin{align}
\label{MSE1}
MSE_h(t)=\E\left\{\left\{\tilde{f}_{nh}(t)-f_0(t)\right\}^2\Bigm|X_1,\dots,X_n\right\}.
\end{align}
Of course, $f_0$ being unknown, this quantity cannot be computed as function of $h$. However, the analogous bootstrap quantity (again using oversmoothing, in the sense that $h_0\asymp n^{-1/9}$) is given by,
\begin{align}
\label{MSE_bootstrap}
MSE_h^*(t)=\E\left\{\left\{\tilde{f}_{nh}^*(t)-\tilde f_{nh_0}(t)\right\}^2\Bigm|(X_1,Y_1),\dots,(X_n,Y_n)\right\},
\end{align}
where $h_0$ is called a ``pilot'' bandwidth. We shall show that (\ref{MSE_bootstrap}) is asymptotically independent of the constant in the pilot bandwidth $h_0$ if we take $h_0\asymp n^{-1/9}$.

We have:
\begin{align}
\label{variance_bias_decomp}
	MSE_h^*(t)&=\E\left\{\left\{\int K_h(t-x)\,\left\{\hat f_n^*(x)-\tilde f_{nh_0}(x)\right\}\,dx\right\}^2\Bigm|(X_1,Y_1),\dots,(X_n,Y_n)\right\}\nonumber\\
&\qquad\qquad	+\E\left\{\left\{\int K_h(t-x)\, \tilde f_{nh_0}(x)\,dx- \tilde f_{nh_0}(t)\right\}^2\Bigm|X_1,\dots,X_n\right\}
\end{align}
For the second term on the right we get:
\begin{align*}
\int K_h(t-x)\,\tilde f_{nh_0}(x)\,dx-\tilde f_{nh_0}(t)
=\tfrac12h^2\tilde f_{nh_0}''(t)\int u^2K(u)\,du+o_p\left(h^2\right),
\end{align*}
so
\begin{align*}
&\E\left\{\left\{\int K_h(t-x)\, \tilde f_{nh_0}(x)\,dx- \tilde f_{nh_0}(t)\right\}^2\Bigm|X_1,\dots,X_n\right\}\\
&=\tfrac14h^4\tilde f_{nh_0}''(t)^2\left\{\int u^2K(u)\,du\right\}^2+o_p\left(h^4\right).
\end{align*}

We have the following result.

\begin{lemma}
\label{lemma:bias_term}
Let the conditions of Theorem \ref{th:limit_SLSE} be satisfied. Moreover, let $h_0=h_{n,0}\sim c_0n^{-1/9}$, as $n\to\infty$. Then
\begin{align*}
&\tilde f_{nh_0}''(t)\stackrel{p}\longrightarrow f_0''(t),\qquad n\to\infty.
\end{align*}
\end{lemma}

\begin{remark}
{\rm
Note that this convergence result does not hold if the pilot bandwidth $h_0$ is of order $n^{-1/5}$. For this reason the method suggested in \cite{SenXu2015}, where the pilot bandwidth is chosen of order $n^{-1/5}$  will not work. Another way out is to use subsampling, as used in
\cite{kim_piet:18:SJS}, but choosing the right subsample size is a rather hard problem.
}
\end{remark}

The proof of Lemma \ref{lemma:bias_term} is given in Section \ref{section:appendix}. 
The lemma suggests, as in \cite{hazelton:96}, to take the pilot bandwidth $h_0=c_0n^{-1/9}$ for some $c_0>0$, taking the optimal order for a bandwidth for estimating the second derivative $f_0''$ in the case that the 4th derivative $f_0^{(4)}(t)$ exists and is not equal to zero. Note that for our example function $f_0(x)=x^2+x/5$ we have $f_0^{(4)}(t)=0$ , so in this case we cannot apply the optimality criterion. The most important fact is, however, that $h_0$ has to tend slower to zero than $n^{-1/5}$, since otherwise the variance of $\tilde f_{n,h_0}''(t)$ does not tend to zero. 

For the first term on the right of (\ref{variance_bias_decomp}) we get:
\begin{align*}
\E\left\{\left\{\int K_h(t-x)\,\left\{\hat f_n^*(x)-\tilde f_{nh_0}(x)\right\}\,dx\right\}^2\Bigm|(X_1,Y_1),\dots,(X_n,Y_n)\right\}
\sim\frac{S_n}{nh}+o_p\left(\frac1{nh}\right),
\end{align*}
where
\begin{align*}
S_n\stackrel{p}\longrightarrow\frac{\s_0^2}{g(t)}\int K(u)^2\,du,
\end{align*}
if $h\sim cn^{-1/5}$ (see (\ref{bootstrap_var}) and the argument using the Lindeberg-Feller central limit theorem part of the proof of Theorem \ref{th:bootstrap_SLSE}).

So, asymptotically, the bandwidth $h$, minimizing (\ref{variance_bias_decomp}) minimizes
\begin{align*}
\frac{\s_0^2}{g(t)nh}\int K(u)^2\,du+\tfrac14h^4\tilde f_0''(t)^2\left\{\int u^2K(u)\,du\right\}^2.
\end{align*}
The minimization of (\ref{MSE1}) leads asymptotically to the same minimization over $h$.

Instead of minimizing (\ref{MSE_bootstrap})  we minimize a  Monte Carlo approximation of (\ref{MSE_bootstrap}):\begin{align}
\label{bootstrap_L2-distance}
B^{-1}\sum_{i=1}^B \left\{\tilde f^{*,i}_{nh}(t)-\tilde f_{nh_0}(t)\right\}^2,
\end{align}
where the $\tilde f^{*,i}_{nh}$, $i=1,\dots,B$ are the estimates in $B$ bootstrap samples.

\begin{figure}[!ht]
\begin{subfigure}[b]{0.3\textwidth}
\includegraphics[width=\textwidth]{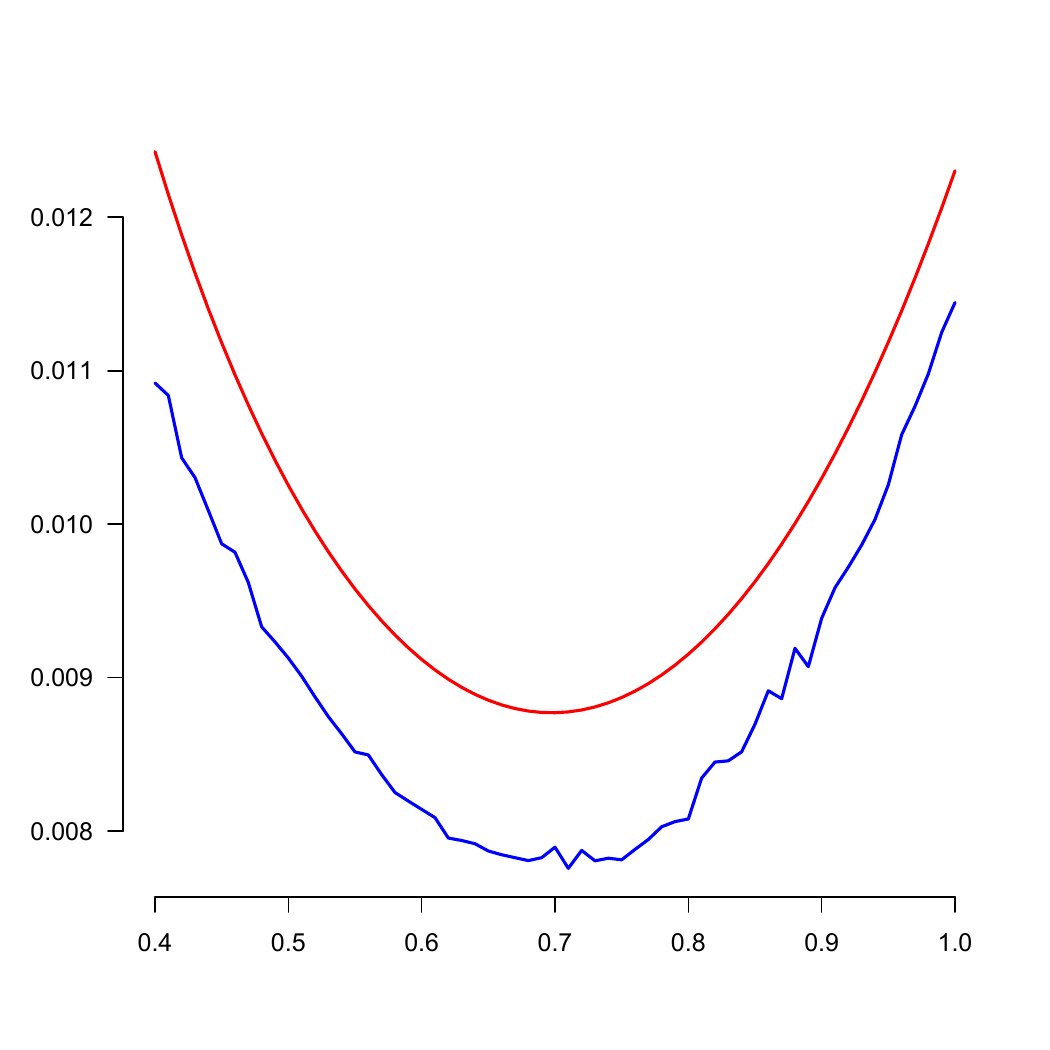}
\caption{}
\label{fig:MSE500}
\end{subfigure}
\begin{subfigure}[b]{0.3\textwidth}
\includegraphics[width=\textwidth]{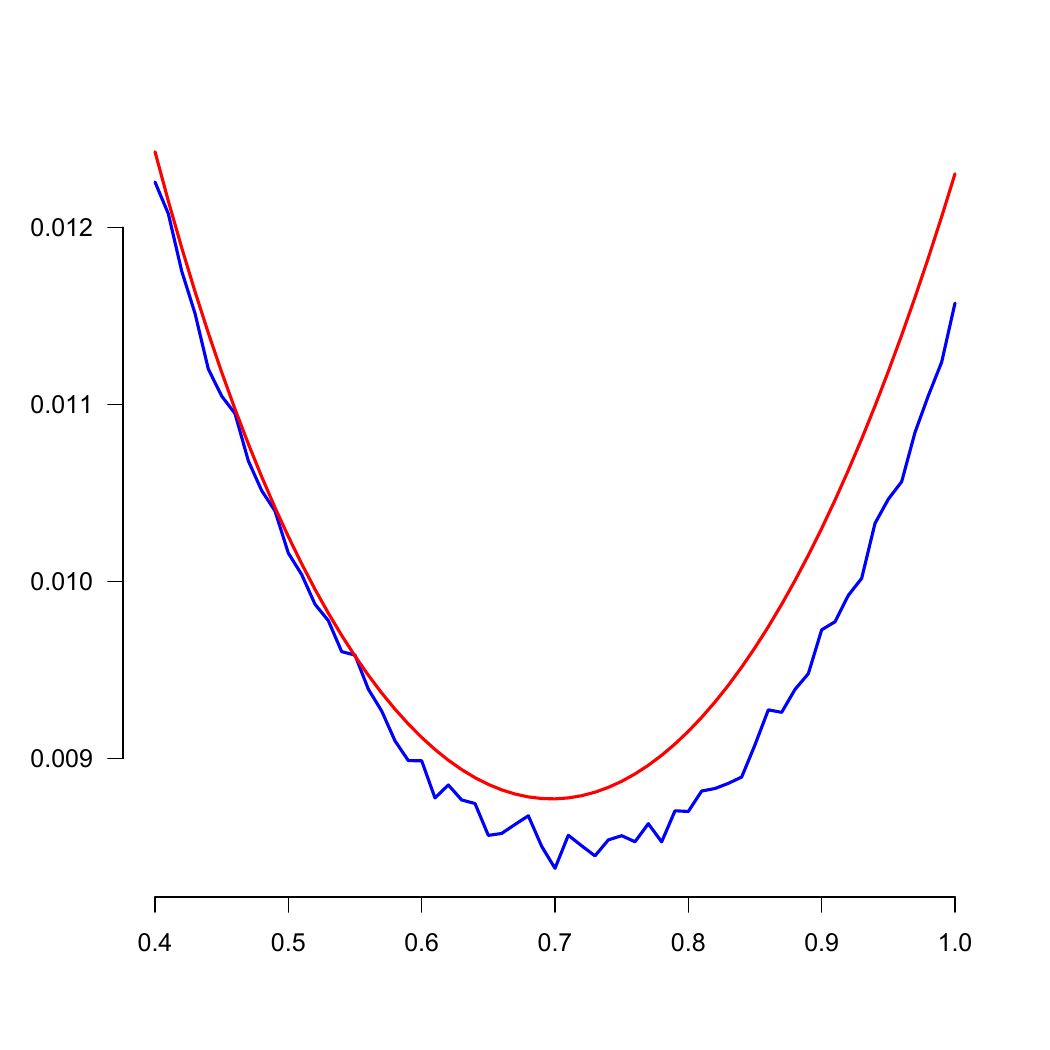}
\caption{}
\label{fig:MISE1000}
\end{subfigure}
\begin{subfigure}[b]{0.3\textwidth}
\includegraphics[width=\textwidth]{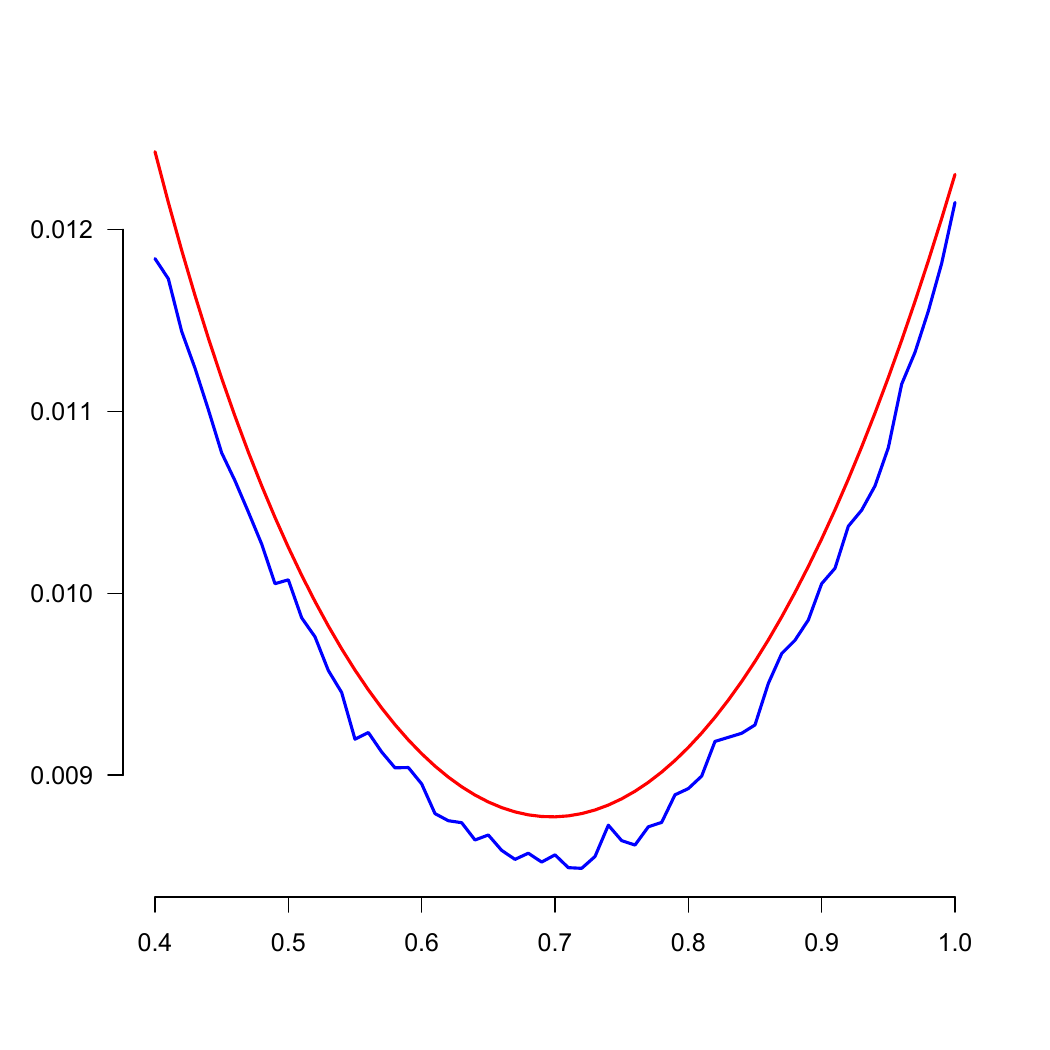}
\caption{}
\label{fig:MISE5000}
\end{subfigure}
\caption{Estimated MISE, for $n=500, 1000$ and $5000$ and $h=cn^{-1/5}$, $c=0.4,0.41,\dots,1$ and asymptotic MISE. The blue curve is based on (\ref{bootstrap_L2-distance_global}), and the red curve on  (\ref{asymptotic_MISE}), $B=10,000$,  where $t_i=0.20,0.21,\dots,0.80$, and $m=60$. (a) Estimates for $n=500$, (b) Estimates for $n=1000$, (c) Estimates for $n=5000$. The asymptotically minimizing $c$ in $h=cn^{-1/5}$ is $c_{min}\approx0.697601$.}
\label{figure:MISE}
\end{figure}

As we are choosing a fixed bandwidth over the range of $t$-values, it is then natural to aim at minimizing the Mean Integrated Squared Error (MISE) as a function of $h$. The asymptotically MISE optimal bandwidth can also be approximated by a smoothed bootstrap experiment, in which case one replaces (\ref{bootstrap_L2-distance}) by
\begin{align}
\label{bootstrap_L2-distance_global}
\widehat{MISE}^*_c=n^{4/5}B^{-1}\sum_{j=1}^B \sum_{i=1}^m\left\{\tilde f^{*,j}_{n,cn^{-1/5}}(t_i)-\tilde f_{nh_0}(t_i)\right\}^2\dd_i,\qquad \dd_i=t_i-t_{i-1}.
\end{align}
for a grid of points $0\le a=t_0<t_1<t_2<\dots<t_m=b\le1$. The latter global minimization produced Figure \ref{figure:MISE}, where we took $a=0.21$ and $b=0.8$. We compare with the plot of the asymptotic MISE as a function of $c$:
\begin{align}
\label{asymptotic_MISE}
AsMISE_c=\frac{\s_0^2}{c}\int K(u)^2\,du\int_{t=a}^b\frac1{g(t)}\,dt+\tfrac14c^4\left\{\int u^2K(u)\,du\right\}^2\int_{t=a}^bf_0''(t)^2\,dt.
\end{align}

In this case we consider
\begin{align}
\label{MISE_bootstrap}
\E\left\{\int_a^b\left\{\tilde{f}_{nh}^*(t)-\tilde f_{nh_0}(t)\right\}^2\,dt\Bigm|(X_1,Y_1),\dots,(X_n,Y_n)\right\},
\end{align}
where $h_0$ is the pilot bandwidth. 
We have:
\begin{align}
\label{variance_bias_decomp_MISE}
&\E\left\{\int_a^b\left\{\tilde{f}_{nh}^*(t)-\tilde f_{nh_0}(t)\right\}^2\,dt\Bigm|(X_1,Y_1),\dots,(X_n,Y_n)\right\}\nonumber\\
&\sim\E\left\{\int_{t=a}^b\left\{\int K_h(t-x)\,\left\{\hat f_n^*(x)-\tilde f_{nh_0}(x)\right\}\,dx\right\}^2\,dt\Bigm|(X_1,Y_1),\dots,(X_n,Y_n)\right\}\nonumber\\
&\qquad\qquad	+\int_{t=a}^b\left\{\int K_h(t-x)\, \tilde f_{nh_0}(x)\,dx- \tilde f_{nh_0}(t)\right\}^2\,dt.
\end{align}
For the second term on the right we get:
\begin{align*}
&\int_{t=a}^b\left\{\int K_h(t-x)\, \tilde f_{nh_0}(x)\,dx-  \tilde f_{nh_0}(t)\right\}^2\,dt\\
&\sim\tfrac14h^4\int_{t=a}^b\tilde f_{nh_0}''(t)^2\,dt\left\{\int u^2K(u)\,du\right\}^2\,dt+o_p\left(h^4\right)
\end{align*}
Since we want $\tilde f_{n,h_0}''$ to be as close as possible to $f_0''$, we suggest to minimize
\begin{align}
\label{MSE_MISE1}
\int_{t=a}^b\left\{\tilde f_{nh_0}''(t)-f_0''(t)\right\}^2\,dt
\end{align}
over $h_0$, which is a direct generalization of the locally optimal choice of $h_0$.

We get:
\begin{align*}
\tilde f_{n,h_0}''(t)-f_0''(t)
= \int K_{h_0}''(t-x)\{\hat f_n(x)-f_0(x)\}\,dx+\int K_{h_0}''(t-x) f_0(x)\,dx-f_0''(t)
\end{align*}
and
\begin{align*}
&\int K_{h_0}''(t-x) f_0(x)\,dx-f_0''(t)\\
&=\int K_{h_0}(t-x) f_0''(x)\,dx-f_0''(t)=\int K_{h_0}(t-x) \left\{f_0''(x)-f_0''(t)\right\}\,dx\\
&=\tfrac12h_0^2 f_0^{(4)}(t)\int u^2K(u)\,du+o\left(h_0^2\right),
\end{align*}
provided a bounded $4$th derivative $f_0^{(4)}(t)$ exists.

This yields:
\begin{align*}
&\int_{t=a}^b\{\tilde f_{n,h_0}''(t)-f_0''(t)\}^2\,dt\\
&\sim \frac{\s_0^2}{nh_0^5}\int K''(u)^2\,du\int_{t=a}^b\frac1{g(t)}\,dt+\tfrac14h_0^4 \left\{\int u^2K(u)\,du\right\}^2
\int_{t=a}^b \left(f_0^{(4)}(t)\right)^2\,dt,
\end{align*}
and minimizing this as a function of $h_0$ gives again $h_0\asymp n^{-1/9}$ if $\int_{t=a}^b \left(f_0^{(4)}(t)\right)^2\,dt\ne0$.

\section{Lake Mendota: yearly number of days frozen}
\label{section:Mendota}
As a real data application, we give confidence intervals for the Lake Mendota data, analyzed in \cite{piet_geurt:14}, Section 1.1. For 157 consecutive years, starting in 1854, the number of days that the lake was frozen was recorded. The idea is that in the presence of global warming, the number of days that the lake is frozen will show a downward trend over the years. It is the first example in \cite{b4:72}.

To apply the methods of the preceding sections, we first rescale the $x$-coordinates to $[0,1]$ by making the transformation
\begin{align*}
X_i:= (X_i-1853)/158=i/158,
\end{align*}
and next by letting $Y_i:=Y_{n-i+1}$. If there is a downward trend in the $Y_i$, there will be an upward trend in the (old) $Y_{n-i+1}$, where $n=157$, and so we can apply the theory of the preceding sections to the new $(X_i,Y_i)$.

As before (in the examples) we take the pilot bandwidth $h_0=0.7n^{-1/9}$, and used the  bandwidth choice of Section \ref{section:bandwidth_choice}, based on (\ref{bootstrap_L2-distance_global}), which gave $c=0.84$ for the optimal bandwidth $h=cn^{-1/5}\approx0.30556$.  The isotonic confidence intervals for this choice of $h$ are shown in Figure \ref{figure:CI_Mendota}. The bootstrap approximation of the MISE as a function of the constant $c$ in the bandwidth choice $h=cn^{-1/5}$ is shown in Figure \ref{figure:MSE_Mendota}.

\begin{figure}[!ht]
\centering
\includegraphics[width=0.6\textwidth]{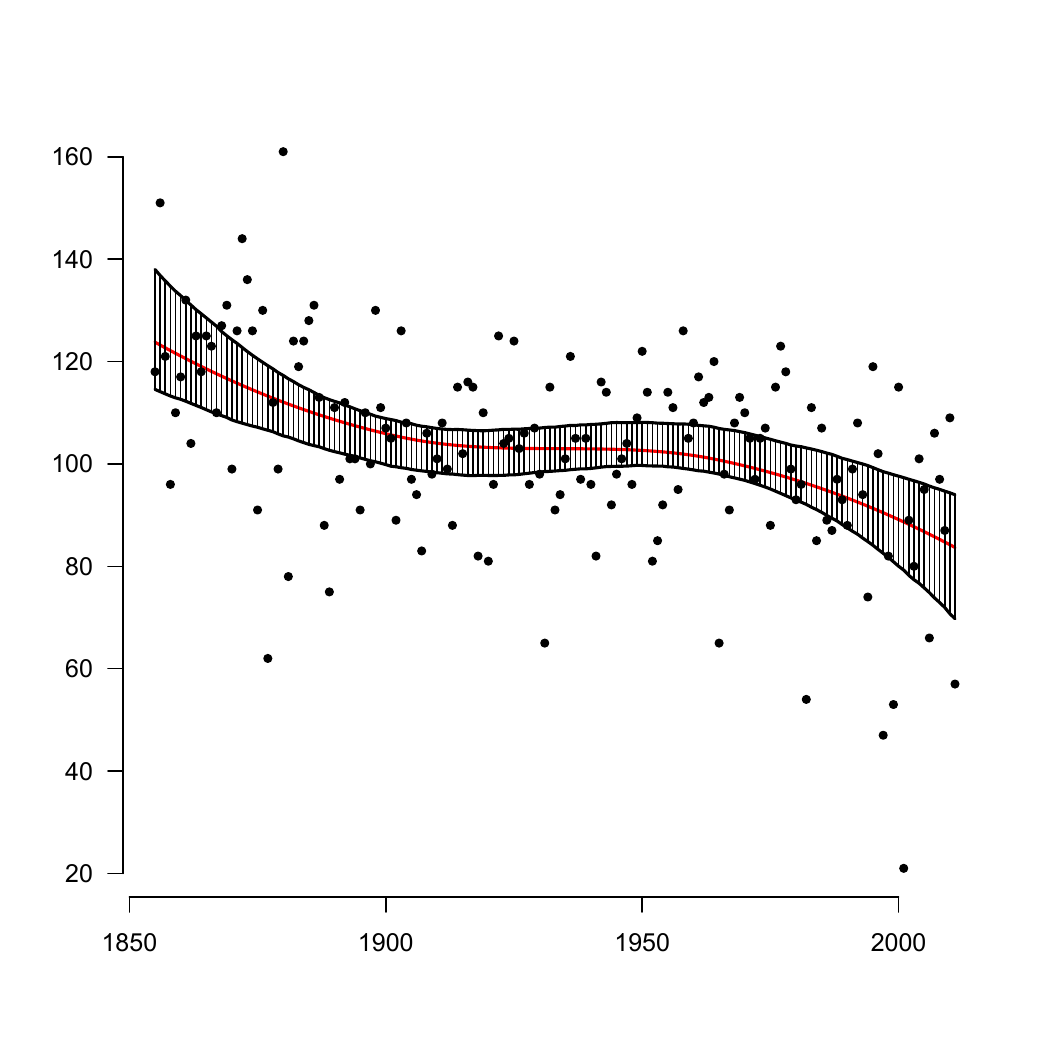}
\caption{95\% confidence intervals for the regression function for the Mendota data. The red curve is the SLSE, with the bandwidth chosen by the method based on the MISE-approximation (\ref{bootstrap_L2-distance_global}).}
\label{figure:CI_Mendota}
\end{figure}

\begin{figure}[!ht]
\centering
\includegraphics[width=0.6\textwidth]{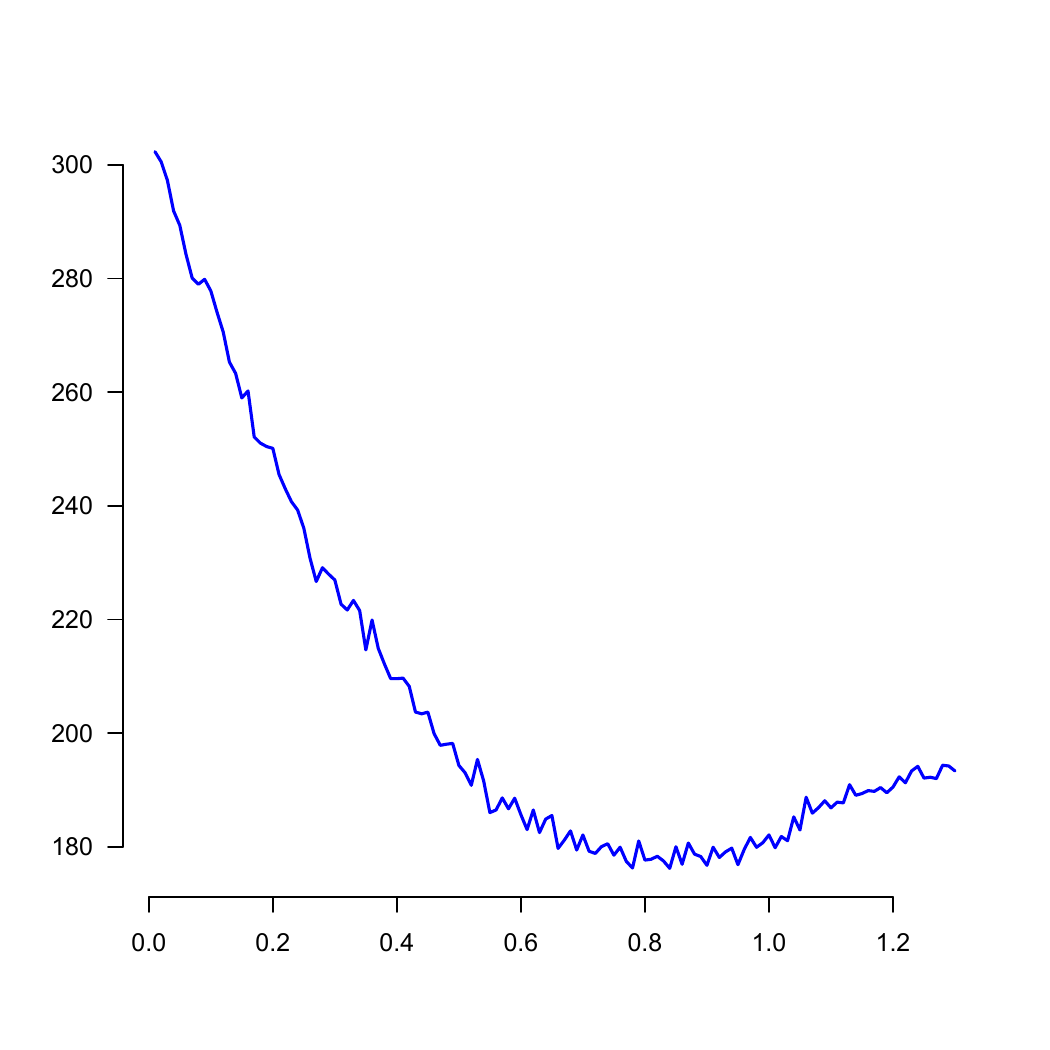}
\caption{The MISE-approximation (\ref{bootstrap_L2-distance_global}), as a function of $c$ for the transformed Mendota data.}
\label{figure:MSE_Mendota}
\end{figure}

\section{Conclusion}
\label{section:conclusion}
We introduce the Smoothed Least Squares Estimator (SLSE) 
\begin{align}
\label{smoothed_LSE}
\tilde f_{nh}(t)=\int K_h(t-x)\hat f_n(x)\,dx
\end{align}
in the monotone regression problem, where $\hat f_n$ is the monotone nonparametric regression estimator, where $K$ is a smooth symmetric kernel with support $[-1,1]$ and $K_h=h^{-1}K(\cdot/h)$. In contrast with $\hat f_n$, the smooth estimator converges at rate $n^{2/5}$, under the conditions of our Theorem 1; the monotone nonparametric regression estimator $\hat f_n$ only converges at rate $n^{1/3}$ in these circumstances. 

We use  the SLSE to construct bootstrap confidence intervals, based on sampling with replacement residuals with respect to an oversmoothed estimator of type (\ref{smoothed_LSE}). The oversmoothing has the effect that the bias is estimated correctly (which is not the case if one uses residuals w.r.t. an estimate, based on a bandwidth of order $n^{-1/5}$) and the bias drops out in the final confidence interval. The idea goes back to similar methods used for NW estimates in \cite{marron:91}.

We compare the isotonic estimates with the NW estimates, suggesting that
the monotone estimates are somewhat more stable if the underlying regression function is monotone. The method extends the construction of confidence intervals for distribution functions in interval censoring models, studied in \cite{SenXu2015} and \cite{kim_piet:18:SJS}. We think the bias problem is solved more efficiently than in \cite{kim_piet:18:SJS}, where undersmoothing or explicit estimation of the bias was suggested, which is also suggested in \cite{hall:92book}.

We describe in Section \ref{section:bandwidth_choice} a method for choosing the bandwidth automatically, correcting the method used in \cite{SenXu2015}. This method is used in Section \ref{section:Mendota} to choose the bandwidth in the classical example of the Lake Mendota data, which was the first example in the book \cite{b4:72}.

Our paper was inspired by the recent paper \cite{moumita:21} for Bayesian confidence intervals in this setting, which will converge at a slower rate and which we will analyze from a bootstrap perspective in a separate paper. We used their example of a regression function in our examples.

All examples in our paper can be recreated using the ${\tt R}$ scripts in \cite{piet_github:21}.

\section{Appendix}
\label{section:appendix}

\begin{proof}[Proof of Theorem \ref{th:limit_SLSE}]
Let $t\in(0,1)$, and let $n$ be sufficiently large, so that $t\in[h,1-h]$, where $h=h_n=cn^{-1/5}$. Then $\tilde f_{nh}$ is represented by the so-called local smooth functional
\begin{align*}
\int K_h(t-x)\,\hat f_n(x)\,dx.
\end{align*}
We now analyze the difference of this functional with the corresponding functional of $f_0$,
\begin{align*}
\int K_h(t-x)\,\left\{\hat f_n(x)-f_0(x)\right\}\,dx.
\end{align*}
Defining $G$ as the distribution function of the $X_i$, we can trivially write this in the following form:
\begin{align*}
&\int K_h(t-x)\,\frac{\hat f_n(x)-f_0(x)}{g(x)}\,dG(x).
\end{align*}

We define
\begin{equation}
\label{def_psi_CS_SML}
\psi_{t,h}(u)=\frac{K_h(t-u)}{g(u)}\,,
\end{equation}
and
$$
\bar\psi_{t,h}(u)=
\left\{\begin{array}{lll}
\psi_{t,h}(\t_i),\,&,\mbox{ if }f_0(u)>\hat f_n(\t_i),\,u\in[\t_i,\t_{i+1}),\\
\psi_{t,h}(s),\,&,\mbox{ if }f_0(s)=\hat f_n(s),\mbox{ for some }s\in[\t_i,\t_{i+1}),\\
\psi_{t,h}(\t_{i+1}),\,&,\mbox{ if }f_0(u)<\hat f_n(\t_i),\,u\in[\t_i,\t_{i+1}),
\end{array}
\right.
$$
where the $\t_i$ are successive points of jump of $\hat f_n$.

We now get:
\begin{align}
\label{decomposition_SLSE}
&\int K_h(t-x)\,\frac{\hat f_n(x)-f_0(x)}{g(x)}\,dG(x)=\int\psi_{t,h}(x)\,\left\{\hat f_n(x)-f_0(x)\right\}\,dG(x)\nonumber\\
&=\int \bar\psi_{t,h}(x)\,\left\{\hat f_n(x)-f_0(x)\right\}\,dG(x)+\int \left\{\psi_{t,h}(x)-\bar\psi_{t,h}(x)\right\}\,\left\{\hat f_n(x)-f_0(x)\right\}\,dG(x)
\end{align}

Let $H_0$ be the distribution function of the pairs $(X_i,Y_i)$, with corresponding empirical distribution function $\H_n$. Then we get for the first term on the last line of in (\ref{decomposition_SLSE})
\begin{align*}
&\int \bar\psi_{t,h}(x)\,\left\{\hat f_n(x)-f_0(x)\right\}\,dG(x)\\
&=\int\bar\psi_{t,h}(x)\,\left\{\hat f_n(x)-y\right\}\,dH_0(x,y)\\
&=\int \bar\psi_{t,h}(x)\,\left\{y-\hat f_n(x)\right\}\,d\left(\H_n-H_0\right)(x,y)\\
&=\int\psi_{t,h}(x)\,\left\{y-\hat f_n(x)\right\}\,d\left(\H_n-H_0\right)(x,y)\\
&\qquad\qquad+\int\left\{\bar\psi_{t,h}(x)-\psi_{t,h}(x)\right\}\,\left\{y-\hat f_n(x)\right\}\,d\left(\H_n-H_0\right)(x,y)\end{align*}
where the second equality follows from the characterization of the LSE $\hat f_n$.

The first term on the last line can be rewritten:
\begin{align*}
&\int\psi_{t,h}(x)\,\left\{y-\hat f_n(x)\right\}\,d\left(\H_n-H_0\right)(x,y)\\
&=\int\psi_{t,h}(x)\,\left\{y-f_0(x)\right\}\,d\left(\H_n-H_0\right)(x,y)\\
&\qquad\qquad+\int\psi_{t,h}(x)\,\left\{f_0(x)-\hat f_n(x)\right\}\,d\left(\H_n-H_0\right)(x,y)\\
&=\int\psi_{t,h}(x)\,\left\{y-f_0(x)\right\}\,d\left(\H_n-H_0\right)(x,y)\\
&\qquad\qquad+\int\psi_{t,h}(x)\,\left\{f_0(x)-\hat f_n(x)\right\}\,d\left(\G_n-G\right)(x),
\end{align*}
where $\G_n$ is the empirical distribution of $X_1,\dots,X_n$.

We now get, under the conditions of the theorem, that the first term on the last line, multiplied by $n^{2/5}$, converges in distribution to a normal distribution with expectation zero and variance
\begin{align}
\label{asymp_var_SLSE}
\s^2=\frac{\s^2_{0}}{c g(t)}\int K(u)^2\,du,
\end{align}
if $h\sim cn^{-1/5}$, as $n\to\infty$, and, by the usual entropy bounds for bounded monotone functions,  the last term is $o_p(n^{-1/5})$. Note that, because of the support of $\psi_{t,h}$, we can restrict ourselves to functions defined on a closed interval $[a,b]\subset(0,1)$, whee $\hat f_n$ is bounded for sufficiently large $n$, because of its consistency on closed subintervals of $(0,1)$.

For the last term of (\ref{decomposition_SLSE}) we follow an argument close to the line of reasoning on p.\ 333 of \cite{piet_geurt:14}. First of all we get, using the Cauchy-Schwarz inequality:
\begin{align}
\label{LSE_ineq1}
\left\|(\hat f_n-f_0)1_{[t-h,t+h]}\right\|_{2,G}=O_p\left(h^{1/2}\|(\hat f_n-f_0)\|_21_{[a,b]}\right),
\end{align}
for a closed subinterval $[a,b]\subset(0,1)$.

 Using this notation, we get, and letting $[a,b]\subset(0,1)$:
\begin{align}
\label{LSE_ineq2}
\left\|(\bar\psi_{t,h}-\psi_{t,h})1_{[a,b]}\right\|_{2,G}=O_p\left(h^{-3/2}\|(\hat f_n-f_0)1_{[a,b]}\|_2\right).
\end{align}
Combining (\ref{LSE_ineq1}) and (\ref{LSE_ineq2}) and using the Cauchy-Schwarz inequality, we obtain:
\begin{align*}
&\left|\int_a^b \left\{\bar\psi_{t,h}(x)-\psi_{t,h}(x)\right\}\,\bigl\{\hat f_n(x)-f_0(x)\bigr\}\,dG(x)\right|\\
&=O_p\left(h^{-1}\left\|(\hat f_n-f_0)1_{[a,b]}\right\|_2^2\right).
\end{align*}
By results in Chapter 9 of \cite{geer:00}, $\|(\hat f_n-f_0)1_{[a,b]}\|_2=O_p(n^{-1/3})$ and if $h\asymp n^{-1/5}$, we obtain from this:
\begin{align*}
&\left|\int_a^b \left\{\bar\psi_{t,h}(x)-\psi_{t,h}(x)\right\}\,\bigl\{\hat f_n(x)-f_0(x)\bigr\}\,dG(x)\right|\\
&=O_p\left(n^{1/5-2/3}\right)=O_p\left(n^{-7/15}\right)=o_p\left(n^{-2/5}\right).
\end{align*}

So we get as conclusion:
\begin{align*}
n^{2/5}\int K_h(t-x)\,\bigl\{\hat f_n(x)-f_0(x)\bigr\}\,dx\stackrel{d}\longrightarrow N(0,\s^2),
\end{align*}
where $N(0,\s^2)$ is a normal distribution with expectation zero and variance (\ref{asymp_var_SLSE}), if $h\sim cn^{-1/5}$, as $n\to\infty$.

Since, under the conditions of the theorem, we have
\begin{align*}
n^{2/5}\left\{\int K_h(t-x)\,f_0(x)\,dx-f_0(t)\right\}=\tfrac12c^2 f_0''(t)\int u^2K(u)\,du+o(1),
\end{align*}
if $h\sim cn^{-1/5}$, as $n\to\infty$, the result  follows.
\end{proof}

\begin{proof}[Proof of Theorem \ref{th:bootstrap_SLSE}]
We follow the set-up of the proof of Theorem \ref{th:limit_SLSE}
and assume $t\in[h,1-h]$. This time $\tilde f_{nh}^*$ is represented by the local smooth functional
\begin{align*}
\int K_h(t-x)\,\hat f_n^*(x)\,dx,
\end{align*}
where $\hat f_n^*$ is the LSE based on a bootstrap sample.
Consider
\begin{align*}
\int K_h(t-x)\,\left\{\hat f_n^*(x)-\tilde f_{nh_0}(x)\right\}\,dx
\end{align*}
and recall that $G$ is the distribution function of the $X_i$. Then:
\begin{align*}
\int K_h(t-x)\,\left\{\hat f_n^*(x)-\tilde f_{nh_0}(x)\right\}\,dx=\int K_h(t-x)\,\frac{\hat f_n^*(x)-\tilde f_{nh_0}(x)}{g(x)}\,dG(x).
\end{align*}

Let $\tilde H_{n,h_0}$ be the distribution function of the pairs $(X_i,Y_i^*)$, induced by uniform sampling with replacement from the centered residuals $\tilde E_i$. Defining $\psi_{t,h}$ as in (\ref{def_psi_CS_SML}), define
$$
\bar\psi_{t,h}^*(x)=
\left\{\begin{array}{lll}
\psi_{t,h}(\t_i^*),\,&,\mbox{ if }\tilde f_{nh_0}(x)>\hat f_n^*(\t_i),\,u\in[\t_i,\t_{i+1}),\\
\psi_{t,h}(s),\,&,\mbox{ if }\tilde f_{nh_0}(s)=\hat f_n^*(s),\mbox{ for some }s\in[\t_i,\t_{i+1}),\\
\psi_{t,h}(\t_{i+1}^*),\,&,\mbox{ if }\tilde f_{nh_0}(x)<\hat f_n^*(\t_i),\,u\in[\t_i,\t_{i+1}),
\end{array}
\right.
$$
where the $\t_i^*$ are successive points of jump of $\hat f_n^*$. We have:
\begin{align*}
\int y\,\bar\psi^*_{t,h}(x)\,d\tilde H_{nh_0}(x,y)=\int \tilde f_{nh_0}(x)\bar\psi^*_{t,h}(x)\,d\tilde H_{n,h_0}(x,y)=\int \tilde f_{nh_0}(x)\bar\psi^*_{t,h}(x)\,d\G_n(x),
\end{align*}
where $\G_n$ is the empirical distribution function of the $X_i$, using the fact that the residuals $\tilde E_i$ are centered. (have mean zero). Hence
\begin{align}
\label{property_ancillary_function}
\int \{y-\tilde f_{nh_0}(x)\}\,\psi_{t,h}(x)\,d\tilde H_{n,h_0}(x,y)=0.\end{align}
By the characterization of the LSE $\hat f_n^*$, we also have
\begin{align*}
\int \{y-\hat f_n^*(x)\}\bar\psi_{t,h}^*(x)\,d\H_n^*(x,y)=0.
\end{align*}

So we get:
\begin{align*}
0&=\int \{y-\hat f_n^*(x)\}\bar\psi^*_{t,h}(x)\,d\H_n^*(x,y)\\
&=\int \{y-\hat f_n^*(x)\}\psi_{t,h}(x)\,d\H_n^*(x,y)+\int\{y-\hat f_n^*(x)\}\bigl\{\bar\psi^*_{t,h}(x)-\psi_{t,h}(x)\bigr\}\,d\H_n^*(x,y)\\
&=\int \{y-\tilde f_{nh_0}(x)\}\psi_{t,h}(x)\,d\bigl(\H_n^*-\tilde H_{n,h_0}\bigr)(x,y)\\
&\qquad+\int \{\tilde f_{nh_0}(x)-\hat f_n^*(x)\}\psi_{t,h}(x)\,d\H_n^*(x,y)\\
&\qquad+\int\{y-\hat f_n^*(x)\}\bigl\{\bar\psi^*_{t,h}(x)-\psi_{t,h}(x)\bigr\}\,d\H_n^*(x,y),
\end{align*}
using (\ref{property_ancillary_function}) to replace $d\H_n^*$ by $d(\H_n^*-\tilde H_{n,h_0})$ in the first  integral after the last equality sign.
This implies
\begin{align}
\label{last_decomp}
&\int K_h(t-x)\bigl\{\hat f_n^*(x)-\tilde f_{nh_0}(x)\bigr\}\,dx\nonumber\\
&=\int\psi_{t,h}(x)\bigl\{\hat f_n^*(x)-\tilde f_{nh_0}(x)\bigr\}\,dG(x)\nonumber\\
&=\int\psi_{t,h}(x)\bigl\{\hat f_n^*(x)-\tilde f_{nh_0}(x)\bigr\}\,dG(x)+\int \{y-\hat f_n^*(x)\}\bar\psi^*_{t,h}(x)\,d\H_n^*(x,y)\nonumber\\
&=\int \{y-\tilde f_{nh_0}(x)\}\psi_{t,h}(x)\,d\bigl(\H_n^*-\tilde H_{n,h_0}\bigr)(x,y)\nonumber\\
&\qquad+\int \{\tilde f_{nh_0}(x)-\hat f_n^*(x)\}\psi_{t,h}(x)\,d\bigl(\G_n-G\bigr)(x)\nonumber\\
&\qquad+\int\{y-\hat f_n^*(x)\}\bigl\{\bar\psi^*_{t,h}(x)-\psi_{t,h}(x)\bigr\}\,d\H_n^*(x,y).
\end{align}

The Lindeberg-Feller central limit theorem implies that
\begin{align*}
n^{2/5}\int \{y-\tilde f_{nh_0}(x)\}\psi_{t,h}(x)\,d\bigl(\H_n^*-\tilde H_{n,h_0}\bigr)(x,y)  \stackrel{\cal D}\longrightarrow N(0,\s^2), 
\end{align*}
given $(X_1,Y_1),\dots,(X_n,Y_n)$, almost surely along sequences $(X_1,Y_1),(X_2,Y_2),\dots$. Note that
\begin{align*}
&n^{2/5}\int \{y-\tilde f_{nh_0}(x)\}\psi_{t,h}(x)\,d\bigl(\H_n^*-\tilde H_{n,h_0}\bigr)(x,y)\\
&=n^{-3/5}\sum_{i=1}^n\{Y_i^*-\tilde f_{nh_0}(x)\}\psi_{t,h}(X_i),
\end{align*}
which has conditional expectation zero and conditional variance
\begin{align}
\label{bootstrap_var}
n^{-1/5}\sum_{i=1}^n\left\{n^{-1}\sum_{j=1}^n \tilde E_j^2\right\}\,\psi_{t,h}(X_i)^2\sim \frac{\s_0^2}{cg(t)}\int K(u)^2\,du,
\end{align}
given $(X_1,Y_1),\dots,(X_n,Y_n)$, if $h\sim cn^{-1/5}$.

Note that
\begin{align*}
n^{-1}\sum_{j=1}^n \tilde E_j^2
\sim n^{-1}\sum_{j=1}^n \left\{Y_j-f_0(X_i)\right\}^2+n^{-1}\sum_{j=1}^n \left\{f_0(X_i)-\tilde f_{nh_0}(X_i)\right\}^2
\longrightarrow \s_0^2,
\end{align*}
almost surely, as $n\to\infty$.

We now turn to the second term after the last equality sign in (\ref{last_decomp}).
By the consistency of $\tilde f_{nh_0}$ and the ensuing conditional consistency of $\hat f_n^*$ on bounded intervals $[a,b]\subset(0,1)$, added to the monotonicity of these functions on such interval, we may assume that
\begin{align*}
{\cal F}_n=\left\{\{\hat f_n^*-\tilde f_{nh_0}\}1_{[a,b]}\right\}
\end{align*}
is of uniformly bounded variation for such an interval $[a,b]$ and hence has entropy with bracketing $H(\e,{\cal F}_n,\G_n)\le c\e^{-1}$  for the $L_2$-distance and some $c>0$, conditionally on $(X_1,Y_1),\dots,(X_n,Y_n)$, along all sequences $(X_1,Y_1),(X_2,Y_2),\dots$. Moreover, the $L_2$-distance $\|\{\hat f_n^*-\tilde f_{nh_0}\}1_{[a,b]}\|_2$ is of order $n^{-1/3}$, just as the $L_2$-distance $\|\{\hat f_n-f_0\}1_{[a,b]}\|_2$. So we find, if $h\asymp n^{-1/5}$,
\begin{align*}
&\int \{\tilde f_{nh_0}(x)-\hat f_n^*(x)\}\psi_{t,h}(x)\,d\bigl(\G_n-G\bigr)(x)=O_p^*\left(h^{-1}n^{-2/3}\right)
=O_p^*\left(n^{-7/15}\right)\\
&=o_p^*\left(n^{-2/5}\right),
\end{align*}
where we add the star $*$ to the $O_p$ and $o_p$ symbol to indicate the conditional meaning of these symbols for the bootstrap samples.

Finally, the third term after the last equality sign in (\ref{last_decomp}) can be written
\begin{align*}
&\int\{y-\hat f_n^*(x)\}\bigl\{\bar\psi^*_{t,h}(x)-\psi_{t,h}(x)\bigr\}\,d\bigl(\H_n^*-\tilde H_{n,h_0}\bigr)(x,y)\\
&\qquad+\int\bigl\{\tilde f_{nh_0}(x)-\hat f_n^*(x)\bigr\}\bigl\{\bar\psi^*_{t,h}(x)-\psi_{t,h}(x)\bigr\}\,d\tilde H_{n,h_0}(x,y)\\
&=\int\{y-\hat f_n^*(x)\}\bigl\{\bar\psi^*_{t,h}(x)-\psi_{t,h}(x)\bigr\}\,d\bigl(\H_n^*-\tilde H_{n,h_0}\bigr)(x,y)\\
&\qquad+\int\bigl\{\tilde f_{nh_0}(x)-\hat f_n^*(x)\bigr\}\bigl\{\bar\psi^*_{t,h}(x)-\psi_{t,h}(x)\bigr\}\,d\G_n(x).
\end{align*}
Using
\begin{align*}
\left|\bar\psi^*_{t,h}(x)-\psi_{t,h}(x)\right|\lesssim h^{-2}\bigl|\hat f_n^*(x)-\tilde f_{nh_0}(x)\bigr|
\end{align*}
for all $x$ in a neighborhood of $t$, we find again that these terms are $o_p(n^{-2/5})$.

Moreover, again conditionally and almost surely,
\begin{align*}
&\int K_h(t-x)\tilde f_{nh_0}(x)\,dx=\tilde f_{nh_0}(t)+\tfrac12 h^2f_0''(t)\int u^2K(u)\,du+o\bigl(h^2\bigr).
\end{align*}
Hence
\begin{align*}
&E\left\{\tilde f_{nh}^*(t)-\tilde f_{nh_0}(t)\bigm|(X_1,Y_1),\dots,(X_n,Y_n)\right\}\\
&=E\left\{\int K_h(t-x)\bigl\{\hat f_n^*(x)-\tilde f_{nh_0}(x)\bigr\}\,dx\bigm|(X_1,Y_1),\dots,(X_n,Y_n)\right\}\\
&\qquad\qquad+\tfrac12 h^2f_0''(t)\int u^2K(u)\,du+o\bigl(h^2\bigr).
\end{align*}
This gives the expression for the mean of the conditional limit distribution of $\tilde f_{nh}^*(t)-\tilde f_{nh_0}(t)$. Note that the bias drops out in the construction of the bootstrap confidence intervals.
\end{proof}

\begin{proof}[Proof of Lemma \ref{lemma:bias_term}]
We have:
\begin{align*}
&\tilde f_{nh_{n,0}}''(t)-h_{n,0}^{-3}\int K''\left((t-x)/h_{n,0}\right)\,f_0(x)\,dx\\
&=h_{n,0}^{-3}\int K''\left((t-x)/h_{n,0}\right)\,\left\{\hat f_n(x)-f_0(x)\right\}\,dx,
\end{align*}
and
\begin{align*}
&\int K''\left((t-x)/h_{n,0}\right)\,\left\{\hat f_n(x)-f_0(x)\right\}\,dx\\
&=\int \frac{K''\left((t-x)/h_{n,0}\right)}{g(x)}\,\left\{\hat f_n(x)-f_0(x)\right\}\,dG(x)\\
&=\int \frac{K''\left((t-x)/h_{n,0}\right)}{g(x)}\,\left\{\hat f_n(x)-y\right\}\,dH_0(x,y),
\end{align*}
where $G$ is the distribution function of the $X_i$ and $H_0$ the distribution function of the pairs $(X_i,Y_i)$.
We now define the function $\psi_{t,h}$ by
\begin{align*}
\psi_{t,h}(x)=\frac{K''\left((t-x)/h\right)}{h^3g(x)}\,.
\end{align*}
and introduce a piecewise constant version $\bar\psi_{t,h}$ of $\psi_{t,h}$ as
\begin{align*}
\bar\psi_{t,h}(x)=
\left\{\begin{array}{lll}
\psi_{t,h}(\t_i),\,&,\mbox{ if }f_0(x)>\hat f_n(\t_i),\,u\in[\t_i,\t_{i+1}),\\
\psi_{t,h}(s),\,&,\mbox{ if }f_0(s)=\hat f_n(s),\mbox{ for some }s\in[\t_i,\t_{i+1}),\\
\psi_{t,h}(\t_{i+1}),\,&,\mbox{ if }f_0(x)<\hat f_n(\t_i),\,u\in[\t_i,\t_{i+1}),
\end{array}
\right.
\end{align*}
where the $\t_i$ are successive points of jump of $\hat f_n$. So we can write:
\begin{align*}
&h_{n,0}^{-3}\int K''\left((t-x)/h_{n,0}\right)\,\left\{\hat f_n(x)-f_0(x)\right\}\,dx\\
&=h_{n,0}^{-3}\int \frac{K''\left((t-x)/h_{n,0}\right)}{g(x)}\,\left\{\hat f_n(x)-y\right\}\,dH_0(x,y))\\
&=\int \psi_{t,h_{n,0}}(x)\,\left\{\hat f_n(x)-y\right\}\,dH_0(x,y)\\
&=\int\left\{\psi_{t,h_{n,0}}(x)-\bar\psi_{t,h_{n,0}}(x)\right\}\,\left\{\hat f_n(x)-y\right\}\,dH_0(x,y)\\
&\qquad\qquad+\int \bar\psi_{t,h_{n,0}}(x)\,\left\{\hat f_n(x)-y\right\}\,dH_0(x,y)\\
&=\int\left\{\psi_{t,h_{n,0}}(x)-\bar\psi_{t,h_{n,0}}(x)\right\}\,\left\{\hat f_n(x)-y\right\}\,dH_0(x,y)\\
&\qquad\qquad+\int \bar\psi_{t,h_{n,0}}(x)\,\left\{\hat f_n(x)-y\right\}\,d\left(H_0-\H_n\right)(x,y),
\end{align*}
using the characterization of the LSE $\hat f_n$ in the last step.
We now get:
\begin{align*}
&\int\left\{\psi_{t,h_{n,0}}(x)-\bar\psi_{t,h_{n,0}}(x)\right\}\,\left\{\hat f_n(x)-y\right\}\,dH_0(x,y)\\
&\int\left\{\psi_{t,h_{n,0}}(x)-\bar\psi_{t,h_{n,0}}(x)\right\}\,\left\{\hat f_n(x)-f_0(x)\right\}\,dG(x)\\
&=O_p\left(h_{n,0}^{-4}n^{-2/3}\right)=O_p\left(n^{-2/9}\right),
\end{align*}
and
\begin{align*}
&\int \bar\psi_{t,h_{n,0}}(x)\,\left\{\hat f_n(x)-y\right\}\,d\left(H_0-\H_n\right)(x,y)
=O_p\left(h_{n,0}^{-3/2}n^{-1/2}\right)=O_p\left(n^{-1/3}\right).
\end{align*}
So the conclusion is:
\begin{align*}
\tilde f_{nh_{n,0}}''(t)-h_{n,0}^{-3}\int K''\left((t-x)/h_{n,0}\right)\,f_0(x)\,dx
=O_p\left(n^{-2/9}\right).
\end{align*}
But under the conditions of Theorem \ref{th:limit_SLSE} we have:
\begin{align*}
h_{n,0}^{-3}\int K''\left((t-x)/h_{n,0}\right)\,f_0(x)\,dx=f_0''(t)+o(1),\qquad n\to\infty.
\end{align*}
\end{proof}

\bibliographystyle{imsart-nameyear}
\bibliography{cupbook}

\end{document}